\documentclass[letterpaper]{amsart}

\usepackage[utf8]{inputenc}
\usepackage{lmodern}
\usepackage[T1]{fontenc}
\usepackage{amssymb}
\usepackage{enumitem}
\usepackage{mathtools}
\usepackage{tikz-cd}

\usepackage[pdfusetitle]{hyperref}

\newlist{enumarabic}{enumerate}{1}
\setlist[enumarabic]{font=\normalfont,label=(\arabic*)}
\newlist{enumroman}{enumerate}{1}
\setlist[enumroman]{font=\normalfont,label=(\roman*)}

\def\arxiv#1{\href{http://arxiv.org/abs/#1}{\texttt{arXiv:#1}}}
\def\doi#1{\href{http://dx.doi.org/#1}{doi:#1}}

\mathtoolsset{mathic=true}
\let\intertext\shortintertext

\theoremstyle{plain}
\newtheorem{theorem}{Theorem}[section]
\newtheorem{proposition}[theorem]{Proposition}
\newtheorem{lemma}[theorem]{Lemma}
\newtheorem{corollary}[theorem]{Corollary}

\theoremstyle{definition}

\newtheorem{remark}[theorem]{Remark}

\theoremstyle{remark}
\newtheorem*{acknowledgements}{Acknowledgements}

\numberwithin{equation}{section}

\DeclareMathAlphabet\mathbfit{OML}{cmm}{b}{it}

\def\Z{\mathbb Z}
\def\Q{\mathbb Q}
\def\R{\mathbb R}
\def\C{\mathbb C}
\def\CP{\mathbb{CP}}

\DeclareMathOperator{\Hom}{Hom}
\DeclareMathOperator{\Ext}{Ext}

\let\MFsm\setminus
\def\setminus{\mathbin{\mkern-1mu\MFsm\mkern-1mu}}

\DeclareMathOperator{\depth}{depth}
\def\at#1{\left.#1\right|}
\def\pair#1{{\langle#1\rangle}}
\def\Hc{H_{c}}
\def\HT{H_{T}}
\def\HTc{H_{T,c}}

\def\AB{AB}
\def\ABc{\AB_{c}}
\def\barAB{{\smash{\overline{AB\mathstrut}}\mathstrut}}

\def\m{\mathfrak m}
\def\hHc{H^{c}}
\def\hHT{H^{T\!}}
\def\hHTc{H^{T,c}}

\def\cf{\emph{cf.}}

\let\epsilon\varepsilon
\let\phi\varphi

\def\citeorbitsone#1{\cite[#1]{AlldayFranzPuppe:orbits1}}
\def\citeorbitsfour#1{\cite[#1]{AlldayFranzPuppe:orbits4}}

\DeclareMathOperator{\rank}{rank}
\DeclareMathOperator{\codim}{codim}

\DeclareMathOperator{\height}{ht}

\def\pp{\mathfrak p}
\def\qq{\mathfrak q}

\def\ttt{\mathfrak t}
\def\B{B_{c}}
\def\barB{\bar B_{c}}
\def\Bc{B_{c}}
\def\BB{\mathcal{B}}
\def\facet{\mathbin{<_{1}}}
\def\PPP{\mathcal P}
\def\AAA{\mathcal{A}}

\def\Hodd{H^{\rm odd}}
\let\kk\R
\def\kktilde{\skew{-2}\tilde\R}
\let\AA\ell
\def\AAopt{;\AA}

\def\AAA{\mathcal A}

\def\piinv#1{X^{#1}}
\def\piinvdot#1{X^{\dot#1}}

\def\PP{\mathbb P}
\DeclareMathOperator{\Bl}{Bl}
\def\BlT{\Bl_{T}\mkern-3mu}
\def\BlL{\Bl_{L}\mkern-2mu}
\def\CP{\mathbb{CP}}

\def\Xf{X_{\rm f}}

\def\tX{\tilde X}
\def\tY{\tilde Y}
\def\TPc{T^{P}\mkern-3mu,c}

\def\tP{t^{P}}

\DeclareMathOperator{\projdim}{proj\,dim}

\begin{document}

\title[A quotient criterion for syzygies]{A quotient criterion for syzygies\\in equivariant cohomology}
\author{Matthias Franz}
\address{Department of Mathematics, University of Western Ontario,
      London, Ont.\ N6A\;5B7, Canada}
\email{mfranz@uwo.ca}
\thanks{The author was supported by an NSERC Discovery Grant.}

\hypersetup{pdfauthor=\authors}


\begin{abstract}
  Let \(X\) be a manifold with an action of a torus~\(T\)
  such that all isotropy groups are connected
  and satisfying some other mild hypotheses.
  We provide a necessary and sufficient criterion for the
  equivariant cohomology~\(\HT^{*}(X)\) with real coefficients
  to be a certain syzygy as module over~\(H^{*}(BT)\).
  It turns out that, possibly after blowing up the non-free part of the action,
  this only depends on the orbit space~\(X/T\)
  together with its stratification by orbit type.
  Our criterion unifies and generalizes results of many authors
  about the freeness and torsion-freeness of equivariant cohomology
  for various classes of \(T\)-manifolds.
\end{abstract}

\maketitle

\section{Introduction}

Let \(R\) be a polynomial ring in \(r\)~variables over a field.
Recall that a finitely generated \(R\)-module~\(M\) is called
a \(j\)-th syzygy
if there is an exact sequence
\begin{equation}
  0\to M\to F_{1}\to \dots \to F_{j}
\end{equation}
with finitely generated free \(R\)-modules~\(F_{1}\),~\ldots,~\(F_{j}\).
The first syzygies are exactly the torsion-free modules,
and the \(r\)-th syzygies the free ones.
Syzygies therefore interpolate between torsion-freeness and freeness.
In~\cite{AlldayFranzPuppe:orbits1} and~\cite{AlldayFranzPuppe:orbits4}, Allday, Puppe and the author
initiated the study of syzygies in the context of torus-equivariant cohomology;
this was extended in~\cite{Franz:nonab} to actions of compact connected Lie groups.

Because restriction to a maximal torus does not change the order of a syzygy
in equivariant cohomology \cite[Prop.~4.2]{Franz:nonab}, we can focus on the torus case.
So let \(T\cong(S^{1})^{r}\) be a torus
and \(R=H^{*}(BT)\) the cohomology of its classifying space
with real coefficients.
Consider the orbit filtration
\begin{equation}
  \label{eq:orbit-filtration}
  \emptyset = X_{-1} \subset
  X_{0} = X^{T} \subset X_{1} \subset \dots \subset X_{r} = X
\end{equation}
of a \(T\)-manifold~\(X\) of finite type.
The equivariant \(i\)-skeleton~\(X_{i}\) is the union of all \(T\)-orbits of dimension at most~\(i\).
The orbit filtration leads to the \emph{Atiyah--Bredon sequence}
\begin{equation}
  \label{eq:3:atiyah-bredon}
  \let\longrightarrow\rightarrow
  0
  \longrightarrow \HT^{*}(X)
  \longrightarrow \HT^{*}(X_0)
  \longrightarrow \HT^{*+1}(X_1, X_0)
  \longrightarrow \cdots
  \longrightarrow \HT^{*+r}(X_r, X_{r-1})
  \longrightarrow 0
\end{equation}
in equivariant cohomology.
The first map in~\eqref{eq:3:atiyah-bredon}
is induced by the inclusion~\(X_{0}\hookrightarrow X\),
and the others are the connecting maps in the long exact sequences
for the triples~\((X_{i+1},X_{i},X_{i-1})\).
We call \(\HT^{*+i}(X_{i}, X_{i-1})\)
the \(i\)-th position of this sequence; \(\HT^{*}(X)\) is at position~\(-1\).
One of the main achievements of~\cite{AlldayFranzPuppe:orbits1} 
was to relate the exactness of the Atiyah--Bredon sequence 
to the syzygy order of~\(\HT^{*}(X)\) \citeorbitsone{Thm.~5.7}:

\begin{theorem}[Allday--Franz--Puppe] 
  \label{thm:3:intro-partial-exactness}
  Let \(j\ge0\).
  The Atiyah--Bredon sequence~\eqref{eq:3:atiyah-bredon}
  is exact at all positions~\(i\le j-2\)
  if and only if
  \(\HT^{*}(X)\) is a \(j\)-th syzygy over~\(R\).  
\end{theorem}

Theorem\nobreakspace \ref {thm:3:intro-partial-exactness}
interpolates between the classical statement that the restriction map~\(\HT^{*}(X)\to\HT^{*}(X_{0})\)
is injective if and only if \(\HT^{*}(X)\) is torsion-free
and a result of Atiyah and Bredon~\cite[Main Lemma]{Bredon:1974}
saying that the whole sequence~\eqref{eq:3:atiyah-bredon}
is exact if \(\HT^{*}(X)\) is free.

Equally important is the case~\(j=2\), which states that the \emph{Chang--Skjelbred sequence}
\begin{equation}
  \let\longrightarrow\rightarrow
  0
  \longrightarrow \HT^{*}(X)
  \longrightarrow \HT^{*}(X_0)
  \longrightarrow \HT^{*+1}(X_1, X_0)
\end{equation}
is exact if and only if \(\HT^{*}(X)\) is a reflexive \(R\)-module (a second syzygy).
In this case one can efficiently compute \(\HT^{*}(X)\) out of data
related only to the fixed points and the \(1\)-dimensional orbits.
This is sometimes called the ``GKM method'' after work
of Goresky--Kottwitz--MacPherson~\cite[Thm.~7.2]{GoreskyKottwitzMacPherson:1998}.
If \(X\) satisfies Poincaré duality, then the second syzygies also characterize
the perfection of the equivariant Poincaré pairing \citeorbitsone{Prop.~5.8}.

One would therefore like to have an easy criterion to decide
whether or not \(\HT^{*}(X)\) is a certain syzygy. In this paper we provide
such a criterion for smooth \(T\)-manifolds,
possibly non-compact and{\slash}or non-orientable.

\smallskip

Let \(X\) be an \(n\)-dimensional smooth \(T\)-manifold. 
We say that the \(T\)-action is \emph{locally standard} if \(n\ge2r\)
and if locally the action looks like the one of \((S^{1})^{r}\) on
\( 
  \C^{r}\times\R^{n-2r}
\) 
that rotates the complex variables in the usual way and acts trivially on the second factor.
This generalizes the definition of `locally standard' in~\cite[Sec.~4.1]{MasudaPanov:2006},
which corresponds to the case~\(n=2r\).
Smooth toric varieties and
many other (open) torus manifolds~\cite{MasudaPanov:2006}, \cite{Masuda:2006}
are locally standard. Our generalization also includes interesting \(T\)-manifolds
whose fixed point set is not discrete, for example
the big polygon spaces introduced in~\cite{Franz:maximal}.

If the action on~\(X\) is not locally standard,
we study a certain blow-up~\(\BlT X\) of it
along fixed-point sets of circles,
see Section\nobreakspace \ref {sec:blowup} for a precise definition.
If all isotropy groups in~\(X\) are connected, then
\(\BlT X\) is again a \(T\)-manifold.
In the presence of disconnected isotropy groups, however,
the blow-up may have orbifold singularities.
Although we expect our results to hold for \(T\)-orbifolds,
we restrict ourselves to actions on manifolds with connected isotropy groups
to keep technicalities within reasonable bounds.

So let us assume for the rest of the introduction
that all isotropy groups occurring in~\(X\) are connected.
We also require \(T\) to act effectively and \(H^{*}(X)\) to be finite-dimensional.
Our first result shows that as far as syzygies are concerned,
it is enough to consider locally standard actions.
Note that if \(T\) acts without fixed points,
then \(\HT^{*}(X)\) cannot be torsion-free, for example
because of Theorem\nobreakspace \ref {thm:3:intro-partial-exactness}.

\begin{proposition}
  \label{thm:syzygy-blowup:intro}
  Let \(X\) be a \(T\)-manifold such that \(X^{T}\ne\emptyset\).
  Then the \(T\)-action on the blow-up~\(\BlT X\) is locally standard.
  Moreover, 
  \(\HT^{*}(X)\) is a \(j\)-th syzygy if and only if \(\HT^{*}(\BlT X)\) is.
\end{proposition}

If the \(T\)-action on~\(X\) is locally standard,
then the quotient~\(X/T\) is a manifold with corners.
Similar to Čech cohomology one can define,
for any (closed) face~\(P\) of~\(X/T\),
a complex~\(\BB^{*}(P)\) with
\begin{equation}
  \BB^{i}(P) =
  \!\! \bigoplus_{\substack{Q\le P \\ \rank Q=i}} \!\! H_{*}(Q)
\end{equation}
and differential
\begin{equation}
  \label{eq:def-d-intro}
  d \sigma = \!\! \sum_{\substack{Q\le O\le P \\ \rank O=i+1}} \!\! \pm(\iota_{QO})_{*}(\sigma)  
\end{equation}
for~\(\sigma\in H_{*}(Q)\subset\BB^{i}(P)\).
Here \(\iota_{QO}\colon Q\hookrightarrow O\) denotes the inclusion and
\(\rank Q\) the common dimension of the orbits in~\(X\) lying
over the interior of~\(Q\).
The signs in~\eqref{eq:def-d-intro} are determined by an ordering of
the facets (corank-\(1\) faces) of~\(X/T\), see Section\nobreakspace \ref {sec:decomposition}.

\begin{theorem}
  \label{thm:condition-syzygy-HBc-intro}
  Let \(X\) be a locally standard \(T\)-manifold, and let \(j\ge0\).
  Then \(\HT^{*}(X)\) is a \(j\)-th syzygy if and only if
  \(H^{i}(\BB^{*}(P))=0\)
  for all faces~\(P\) of~\(X/T\) and all \(i>\max(\rank P-j,0)\).
\end{theorem}

This characterization unifies and extends results of many authors
concerning the freeness and torsion-freeness of
torus-equivariant cohomology.
This includes work of
Barthel--Brasselet--Fieseler--Kaup~\cite{BarthelBrasseletFieselerKaup:2002},
Masuda--Panov \cite{MasudaPanov:2006},
Masuda~\cite{Masuda:2006} and
Goertsches--Rollenske~\cite{GoertschesRollenske:2011}.
We also generalize a result of Bredon~\cite{Bredon:1974} about the cohomology of~\(X/T\)
and one of Ayzenberg--Masuda--Park--Zeng~\cite{AyzenbergEtAl:2014}
about the equivariant cohomology of certain torus manifolds,
and we recover the calculation of \(\Ext\)~modules of Stanley--Reisner rings done by
Mustaţă~\cite{Mustata:2000} and
Yanagawa~\cite{Yanagawa:2000}
as well as Munkres' formula for the depth of such rings \cite{Munkres:1984}.

Compact orientable \(T\)-manifolds of dimension~\(2r\)
have received a lot of attention so far,
for example smooth complete toric varieties
or torus manifolds.
We can now also explain why these spaces do not provide
interesting examples of syzygies.

\begin{corollary}
  \label{thm:torus-manifold-free-intro}
  Let \(X\) be a compact orientable \(T\)-manifold of dimension~\(2r\).
  Then\linebreak\(\HT^{*}(X)\) is torsion-free
  if and only if it is free over~\(R\).
\end{corollary}

\smallskip

The paper is organized as follows:
In Section\nobreakspace \ref {sec:prelim} we review background material, and in Section\nobreakspace \ref {sec:syzygies-fixedpoints} we characterize syzygies
by a depth condition on the equivariant cohomology
of fixed point sets of subtori.
Then we study
blow-ups of characteristics manifolds (Section\nobreakspace \ref {sec:blowup}).
Theorem\nobreakspace \ref {thm:condition-syzygy-HBc-intro}
is proved in Section\nobreakspace \ref {sec:decomposition},
based on a certain decomposition of the Atiyah--Bredon sequence for compact supports.
Multiplicative aspects are discussed in Section\nobreakspace \ref {sec:product}.
We present several applications in Section\nobreakspace \ref {sec:previous-work},
commenting on the relation to previous work;
in Section\nobreakspace \ref {sec:3:examples} we finally
illustrate the computational power of Theorem\nobreakspace \ref {thm:condition-syzygy-HBc-intro}
by revisiting and generalizing the examples
that were considered in~\cite{AlldayFranzPuppe:orbits1}.

\begin{acknowledgements}
  It is a pleasure to acknowledge the fruitful collaboration with Chris Allday and
  Volker Puppe \cite{AlldayFranzPuppe:orbits1},~\cite{AlldayFranzPuppe:orbits4}
  on which the present work is based.
  I am also indebted to Volker Puppe for several suggestions which have
  led to an improved presentation and to Sergio Chaves, Oliver Goertsches, Hiroaki Ishida
  and Michael Wiemeler for stimulating discussions.
  I finally thank the anonymous referees for their meticulous work and
  their valuable suggestions and corrections, in particular for providing the example
  of a regular, but not locally standard \(T\)-action in Section~\ref{sec:blowup} and
  for bringing Munkres' formula 
  to my attention.
\end{acknowledgements}

\section{Preliminaries}
\label{sec:prelim}

\subsection{Standing assumptions}

All tensor products are over~\(\kk\)
and all cohomology is taken with real coefficients unless indicated otherwise.
In particular, \(R=H^{*}(BT)\) is a polynomial ring
in \(r\)~generators of degree~\(2\) with real coefficients.
Here and elsewhere \(H^{*}(-)\) denotes Alexander--Spanier cohomology.
We say that a space~\(Y\) is \emph{acyclic} if \(\tilde H^{*}(Y)=0\).
Note that the empty set is not acyclic since \(\tilde H_{-1}(\emptyset)=\kk\).

From the next section on, we will only consider \(T\)-manifolds.
However, in this section we more generally allow the same \(T\)-spaces~\(X\)
as in~\cite{Franz:nonab}. This means that \(X\) is a closed \(T\)-stable
subset of a \(T\)-manifold or \(T\)-orbifold
such that \(H^{*}(X)\) is a finite-dimensional \(\kk\)-vector space.
Also, only finitely many subtori of~\(T\) are allowed to occur
as the identity component of an isotropy group in~\(X\).
This latter condition is automatically satisfied
if \(X\) is a manifold or locally orientable orbifold
\cite[Thm.~7.7]{Franz:nonab}.

For any subgroup~\(K\subset T\), we write \(X^{K}\subset X\) for the closed subset of \(K\)-fixed points.
We call \(K\) an \emph{isotropy subtorus occurring in~\(X\)} if it is
the identity component of the isotropy group~\(T_{x}\) of some~\(x\in X\).

\subsection{Depth and syzygies}
\label{sec:depth}

In this section,
\(R\) denotes a polynomial ring over some field in \(r\)~indeterminates of positive degrees,
and \(\m\lhd R\) its maximal homogeneous ideal.

Since \(R\) is regular, one has
\begin{equation}
  \label{eq:depth-dim-Rp}
  \depth_{R_{\pp}}R_{\pp} = \dim R_{\pp}=\height\pp
\end{equation}
for all prime ideals~\(\pp\lhd R\).
Here \(\height\pp\) denotes the height of~\(\pp\),
that is, the maximal length of a chain of prime ideals contained in~\(\pp\).
We observe that if \(\pp\) is generated by \(m\) linearly independent linear polynomials,
then
\begin{equation}
  \label{eq:height-linear}
  \height\pp=m.
\end{equation}
(It cannot be smaller because there is an obvious chain of prime ideals of length~\(m\)
ending with~\(\pp\). And if there were a longer chain,
we could similarly extend it to a chain in~\(R\) of length greater than~\(r\),
thus contradicting \(\dim R=r\).)

Let \(M\) be a finitely generated graded \(R\)-module.
An important role in our arguments will be played by the \(\Ext\)~modules
\begin{equation}
  \Ext_{R}^{*}(M,R).
\end{equation}
(Note that because \(M\) is finitely generated,
the graded \(\Ext\) and the ungraded \(\Ext\) coincide, \cf~\cite[p.~33]{BrunsHerzog:1998}.)
Moreover, for any (possibly ungraded) ring extension~\(S\supset R\)
which is flat over~\(R\) one has
\begin{align}
  \Hom_{S}(S\otimes_{R}M,S) &= S\otimes_{R}\Hom_{R}(M,R) \\
  \shortintertext{and therefore}
  \label{eq:ext-ring-ext}
  \Ext_{S}^{*}(S\otimes_{R}M,S) &= S\otimes_{R}\Ext_{R}^{*}(M,R).
\end{align}

The depth of~\(M\) is the common length of all maximal \(M\)-regular sequences in~\(\m\)
(which is \(\infty\) for~\(M=0\));
we will often use that it can equivalently be expressed as
\begin{equation}
  \label{eq:depth-Ext}
  \depth M = \min\{\, i \,|\, \Ext_{R}^{r-i}(M,R)\ne0 \,\},
\end{equation}
\cf~\cite[Prop.~A1.16]{Eisenbud:2005}.
Note that \(\depth M=r\) if and only if \(M\ne0\) is free.
A formula analogous to~\eqref{eq:depth-Ext} exists for finitely generated modules~\(N\)
over any localization~\(R_{\pp}\) of~\(R\),
\begin{equation}
  \label{eq:depth-Ext-local}
  \depth_{R_{\pp}} N = \min\{\, i \,|\, \Ext_{R_{\pp}}^{\bar r-i}(N,R_{\pp})\ne0 \,\},
\end{equation}
where \(\bar r=\depth R_{\pp}\):
Given that \(N\) has finite projective dimension,
\eqref{eq:depth-Ext-local} is by the Auslander--Buchsbaum formula equivalent to the statement
\begin{equation}
  \label{eq:projdim-Ext-local}
  \projdim_{R_{\pp}} N = \max\{\, i \,|\, \Ext_{R_{\pp}}^{i}(N,R_{\pp})\ne0 \,\}.
\end{equation}
(One clearly has ``\(\ge\)''. Moreover, \eqref{eq:projdim-Ext-local}
becomes an equality if one substitutes
the residue field~\(k\) for~\(R_{\pp}\) as the second argument of~\(\Ext\),
\cf~\cite[Prop.~1.3.1]{BrunsHerzog:1998}.
Since \(R_{\pp}\) surjects onto~\(k\), \eqref{eq:projdim-Ext-local} 
follows from the long exact sequence for~\(\Ext\).)

The definition of a syzygy has been given in the introduction.
We will need that \(M\) is a \(j\)-th syzygy over~\(R\)
if and only if
\begin{equation}
  \label{eq:condition-syzygy}
  \depth_{R_{\pp}}M_{\pp} \ge \min(j,\depth_{R_{\pp}} R_{\pp}) 
\end{equation}
holds for all prime ideals~\(\pp\lhd R\) \cite[Sec.~16E]{BrunsVetter:1988}.
Note that the depth of~\(R_{\pp}\) can be replaced by its dimension
according to~\eqref{eq:depth-dim-Rp}.
Together with~\eqref{eq:depth-Ext-local} and the additivity of the \(\Ext\)~functor,
the criterion~\eqref{eq:condition-syzygy} shows
that a direct sum~\(M\oplus M'\) is a \(j\)-th syzygy if and only if both \(M\)~and \(M'\) are.

\subsection{The Atiyah--Bredon complex}

We write
\(\AB^{*}(X)\) or, if necessary, \(\AB_{T}^{*}(X)\), for the non-augmented Atiyah--Bredon complex
\begin{equation}
  \label{eq:3:atiyah-bredon-nonaug}
  \let\longrightarrow\rightarrow
  \HT^{*}(X_0)
  \longrightarrow \HT^{*+1}(X_1, X_0)
  \longrightarrow \cdots
  \longrightarrow \HT^{*+r}(X_r, X_{r-1})
\end{equation}
obtained by removing the leading term~\(\HT^{*}(X)\) from the sequence~\eqref{eq:3:atiyah-bredon}.
By a result of Allday--Franz--Puppe~\citeorbitsone{Thm.~5.1},
the cohomology of the Atiyah--Bredon complex depends only
on the equivariant homology~\(\hHT_{*}(X)\)
as defined in~\citeorbitsone{Sec.~3.3}:

\begin{theorem} 
  \label{thm:3:exthab}
  For any~\(i\ge0\),
  \begin{equation*}
    H^{i}(\AB^{*}(X))=\Ext^{i}_{R}(\hHT_{*}(X),R).
  \end{equation*}
\end{theorem}

More important than the definition of equivariant homology 
is for us that it is related
to equivariant cohomology via equivariant Poincaré duality:
If \(X\) is a compact orientable \(T\)-manifold of dimension~\(n\), then
\begin{equation}
  \HT^{*}(X) \cong \hHT_{n-*}(X)
\end{equation}
as \(R\)-modules \citeorbitsone{Prop.~3.7}.
More generally, for any \(T\)-manifold~\(X\) one has
\begin{equation}
  \label{eq:PD-c}
  \HT^{*}(X) \cong \hHTc_{n-*}(X;\kktilde)
\end{equation}
where now \(\hHTc_{*}(X;\kktilde)\) denotes the equivariant homology
of~\(X\) with closed supports and twisted coefficients \citeorbitsfour{Sec.~2.6, Prop.~3.2}.
In this case one can also define an Atiyah--Bredon complex~\(\ABc^{*}(X;\kktilde)\)
for equivariant cohomology with compact supports and twisted coefficients:
\begin{equation}
  \label{eq:3:atiyah-bredon-nonaug-c}
  \let\longrightarrow\rightarrow
  \HTc^{*}(X_0;\kktilde)
  \longrightarrow \HTc^{*+1}(X_1, X_0;\kktilde)
  \longrightarrow \cdots
  \longrightarrow \HTc^{*+r}(X_r, X_{r-1};\kktilde).
\end{equation}
It is related to~\(\hHTc_{*}(X;\kktilde)\) through an analogue of Theorem\nobreakspace \ref {thm:3:exthab},
\begin{equation}
  \label{eq:AB-Ext-R-c}
  H^{i}(\ABc^{*}(X;\kktilde))=\Ext^{i}_{R}(\hHTc_{*}(X;\kktilde),R)
\end{equation}
for any~\(i\ge0\) \citeorbitsfour{Cor.~4.6}.

Combining these results with~\eqref{eq:depth-Ext}, we obtain:

\begin{proposition} $ $
  \label{thm:depth-ABc}
  \begin{enumroman}
  \item
    For any \(T\)-space~\(X\),
    \begin{equation*}
      \depth \hHT_{*}(X) = \min\{\, i \,|\, H^{r-i}(\AB^{*}(X))\ne0 \,\}.
    \end{equation*}
  \item \label{thm:depth-ABc-hommf}
    If \(X\) is a \(T\)-manifold, then
    \begin{equation*}
      \depth \HT^{*}(X) = \min\{\, i \,|\, H^{r-i}(\ABc^{*}(X;\kktilde))\ne0 \,\}.
    \end{equation*}
  \end{enumroman}
\end{proposition}

\section{Syzygies and fixed point sets}
\label{sec:syzygies-fixedpoints}

Let \(K\subset T\) be a subtorus and set \(L=T/K\). We write \(R_{K}=H^{*}(BK)\) and~\(R_{L}=H^{*}(BL)\).
(This naming scheme is adopted for other tori as well.)
The projection~\(T\to L\) induces a morphism of algebras~\(R_{L}\to R\), so that
\(R\) is canonically an \(R_{L}\)-module. Identifying \(L\) with some torus complement to~\(K\) in~\(T\),
we get \emph{non-canonical} isomorphisms \(T\cong K\times L\) and \(R\cong R_{K}\otimes R_{L}\).
Note that \(R\) is free as \(R_{L}\)-module.

\begin{lemma}
  \label{thm:reduction-effective}
  Assume that the subtorus~\(K\subset T\) acts trivially on~\(X\), so that we get an induced action of~\(L=T/K\) on~\(X\).
  \begin{enumroman}
  \item \label{thm:reduction-effective-HT}
    There are isomorphisms of \(R\)-algebras
    \begin{equation*}
     \HT^{*}(X)=R\otimes_{R_{L}}H_{L}^{*}(X)\cong R_{K}\otimes H_{L}^{*}(X). 
    \end{equation*}
    The first isomorphism is canonical. The second one depends on the chosen
    splitting~\(T\cong K\times L\) via the isomorphism~\(R\cong R_{K}\otimes R_{L}\).
  \item \label{thm:reduction-effective-depth}
    \(\depth_{R} \HT^{*}(X)=\depth_{R_{L}}H_{L}^{*}(X)+\dim K\).
  \end{enumroman}
\end{lemma}

Note that one may always pass to an almost effective action
by dividing out the largest trivially operating subtorus of~\(T\).
(A torus action on~\(X\)
is called \emph{almost effective} if only finitely many group elements act trivially.)

\begin{proof}
  The first assertion is well-known, see for example~\citeorbitsfour{Prop.~2.6},
  and it implies
  \begin{equation}
    \Ext_{R}^{i}(\HT^{*}(X),R) 
    = R\otimes_{R_{L}}\Ext_{R_{L}}^{i}(H_{L}^{*}(X),R_{L}).
  \end{equation}
  for any~\(i\ge0\) by~\eqref{eq:ext-ring-ext}.
  As \(R\) is free over~\(R_{L}\), we conclude
  \begin{equation}
    \Ext_{R}^{i}(\HT^{*}(X),R) = 0
    \quad\Longleftrightarrow\quad
    \Ext_{R_{L}}^{i}(H_{L}^{*}(X),R_{L}) = 0.
  \end{equation}
  Together with~\eqref{eq:depth-Ext} this proves the second claim
  since \(\dim R=\dim R_{L}+\dim K\).
\end{proof}

Let \(\pp\lhd R\) be a prime ideal, and let \(\qq\subset\pp\)
be the prime ideal generated by \(\pp\cap H^{2}(BT;\Q)\).
Then \(\qq\) is the kernel of the restriction map~\(H^{*}(BT)\to H^{*}(BK)\)
for some subtorus~\(K\subset T\) with quotient~\(L=T/K\).
As above, we identify \(L\) with a torus complement to~\(K\) in~\(T\).

\begin{lemma}
  \label{thm:prop-Rpp}
  Keeping this notation, one has for any~\(i\ge0\) that
  \begin{equation*}
    \Ext_{R_{\pp}}^{i}(\HT^{*}(X)_{\pp},R_{\pp}) = 0
    \quad\Longleftrightarrow\quad
    \Ext_{R_{L}}^{i}(H_{L}^{*}(X^{K}),R_{L}) = 0.
  \end{equation*}
\end{lemma}

\begin{proof}
  We start by observing that 
  the \(R_{L}\)-module~\(R_{\pp}\) is faithfully flat.
  This is
  because \(R\cong R_{K}\otimes R_{L}\) is a free \(R_{L}\)-module,
  localization is exact and the localization map
  \begin{equation}
    R\otimes_{R_{L}}N \cong R_{K}\otimes N\to (R_{K}\otimes N)_{\pp} \cong R_{\pp}\otimes_{R_{L}}N
  \end{equation}
  is injective for any \(R_{L}\)-module~\(N\),
  \cf~\cite[Lemma~1.2]{ChangSkjelbred:1974}.
  (Write an~\(a\notin\pp\) as a sum of homogeneous elements
  of the form~\(a'\otimes a''\in R_{K}\otimes R_{L}\).
  Then it must have a 
  component in~\(R_{K}\otimes\kk\), so that
  multiplication by~\(a\) is injective on~\(R_{K}\otimes N\).)  

  By Lemma\nobreakspace \ref {thm:reduction-effective}\,\ref{thm:reduction-effective-HT} and
  the localization theorem in equivariant cohomology,
  \cf~\citeorbitsone{Thm.~2.5},
  we have
  \begin{equation}
    \HT^{*}(X)_{\pp} = \HT^{*}(X^{K})_{\pp}
    = \bigl( R\otimes_{R_{L}} H_{L}^{*}(X^{K}) \bigr)_{\pp}
    = R_{\pp} \otimes_{R_{L}} H_{L}^{*}(X^{K}),
  \end{equation}
  hence
  \begin{equation}
    \Ext_{R_{\pp}}^{i}(\HT^{*}(X)_{\pp},R_{\pp})
    = R_{\pp}\otimes_{R_{L}} \Ext_{R_{L}}^{i}(H_{L}^{*}(X^{K}),R_{L})
  \end{equation}
  by~\eqref{eq:ext-ring-ext}.
  Thus, the claim follows from faithful flatness.
\end{proof}

\begin{proposition} \label{thm:syzygy-depth-fixedpoints} $ $
  \begin{enumroman}
  \item \label{thm:syzygy-depth-1}
    For any subtorus~\(K\subset T\), 
    \begin{equation*}
      \depth_{R}\HT^{*}(X^{K}) 
      \ge \depth_{R}\HT^{*}(X). 
    \end{equation*}
  \item \label{thm:syzygy-depth-2}
    \(\HT^{*}(X)\) is a \(j\)-th syzygy if and only if
    \begin{equation*}
      \depth_{R_{L}} H_{L}^{*}(X^{K}) \ge \min(j,\dim L)
    \end{equation*}
    for any subtorus~\(K\subset T\) with quotient~\(L=T/K\).
    If \(X\) is a \(T\)-manifold, then it suffices to look at the isotropy subtori~\(K\) occurring in~\(X\).
  \item \label{thm:syzygy-XK}
    If \(\HT^{*}(X)\) is a \(j\)-th syzygy over~\(R\), then
    so is \(H_{L}^{*}(X^{K})\) over~\(R_{L}\)
    for any subtorus~\(K\subset T\) with quotient~\(L=T/K\).
  \end{enumroman}
\end{proposition}

\begin{proof}
  For the first claim
  we use the notation introduced before Lemma\nobreakspace \ref {thm:prop-Rpp}
  with~\(\pp=\qq=\ker(H^{*}(BT)\to H^{*}(BK))\).
  In terms of the identity~\eqref{eq:depth-Ext},
  our assumption is
  \begin{equation}
    \Ext_{R}^{i}(\HT^{*}(X),R) = 0
  \end{equation}
  for~\(i>r-\depth\HT^{*}(X)=\dim L-(\depth\HT^{*}(X)-\dim K)\),
  hence also
  \begin{equation}
    \Ext_{R_{\pp}}^{i}(\HT^{*}(X)_{\pp},R_{\pp})=\bigl(\Ext_{R}^{i}(\HT^{*}(X),R)\bigr)_{\pp} = 0.
  \end{equation}
  By Lemma\nobreakspace \ref {thm:prop-Rpp},
  this implies
  \begin{equation}
    \depth_{R_{L}}H_{L}^{*}(X^{K}) \ge \depth_{R}\HT^{*}(X)-\dim K.
  \end{equation}
  From Lemma\nobreakspace \ref {thm:reduction-effective}\,\ref{thm:reduction-effective-depth}
  we now conclude
  \begin{equation}
    \depth_{R}\HT^{*}(X^{K})=\dim K+\depth_{R_{L}}H_{L}^{*}(X^{K}) \ge \depth_{R}\HT^{*}(X).
  \end{equation}

  We now turn to the equivalence in part~\ref{thm:syzygy-depth-2}.

  \(\Leftarrow\):
  We verify condition~\eqref{eq:condition-syzygy}. For a given prime ideal~\(\pp\lhd R\),
  we choose~\(\qq\), \(K\) and~\(L\) as done before Lemma\nobreakspace \ref {thm:prop-Rpp}.
  Recall that \(\dim R_{\qq}\) equals the height of~\(\qq\), which in our case is \(\dim L\)
  by~\eqref{eq:height-linear}.

  Combining our assumption with~\eqref{eq:depth-Ext},
  we see that \(\Ext_{R_{L}}^{i}(H_{L}^{*}(X),R_{L})\)
  vanishes for all \(i > \max(\dim L-j,0)\), whence
  \begin{equation}
    \label{eq:Ext-Rpp-0}
    \Ext_{R_{\pp}}^{i}(\HT^{*}(X)_{\pp},R_{\pp}) = 0
  \end{equation}
  by Lemma\nobreakspace \ref {thm:prop-Rpp}.
  Since \(\dim L= 
  \height\qq\le\height\pp=\dim R_{\pp}=\vcentcolon\bar r\),
  we in particular have \eqref{eq:Ext-Rpp-0}
  for all~\(i > \max(\bar r-j,0)\),
  which by~\eqref{eq:depth-Ext} means
  \begin{equation}
    \depth_{R_{\pp}}\HT^{*}(X)_{\pp} \ge \min(j,\bar r).
  \end{equation}
  Hence \(\HT^{*}(X)\) is a \(j\)-th syzygy by~\eqref{eq:condition-syzygy}. 

  \(\Rightarrow\): One reverses the argument with~\(\pp=\qq=\ker(H^{*}(BT)\to H^{*}(BK))\),
  using 
  that \(\dim L=\dim R_{\pp}\) in this case.

  Consider next a \(T\)-manifold~\(X\) satisfying the depth condition for
  all isotropy subtori occurring in it, and let \(K\subset T\) be a subtorus that does not occur.
  Each connected component~\(Y\) of~\(X^{K}\) is a \(T\)-manifold.
  Let \(T/\tilde K\) be the principal orbit type of~\(Y\).
  Then the identity component~\(K'\) of~\(\tilde K\) occurs in~\(X\) and properly contains \(K\),
  and \(Y\) is a connected component of~\(X^{K'}\).
  Write \(\dim K'=\dim K+s\) and \(L'=T/K'\).
  By assumption and Lemma\nobreakspace \ref {thm:reduction-effective}\,\ref{thm:reduction-effective-depth},
  we have
  \begin{align}
    \depth_{R_{L}} H_{L}^{*}(Y) &\ge \depth_{R_{L}} H_{L}^{*}(X^{K'})
    = \depth_{R_{L'}} H_{L'}^{*}(X^{K'}) + s \\
    \notag
    &\ge \min(j,\dim L') + s \ge \min(j,\dim L) .
  \end{align}
  Because this holds for any connected component~\(Y\) of~\(X^{K}\),
  the depth condition is satisfied for~\(K\), too.

  The third claim is a consequence of the second:
  The subtori of~\(L=T/K\)
  correspond bijectively to those of~\(T\) containing~\(K\).
  If \(K''\subset L\) corresponds to~\(K'\subset T\), then \(L''=L/K''\) is isomorphic to~\(L'=T/K'\),
  and \((X^{K})^{K''}=X^{K'}\).
  Given that \(\HT^{*}(X)\) is a \(j\)-th syzygy over~\(R\), we get
  \begin{align}
    \depth_{R_{L''}}H_{L''}^{*}\bigl((X^{K})^{K''}\bigr)
    &= \depth_{R_{L'}}H_{L'}^{*}(X^{K'}) \\
    \notag
    &\ge \min(j,\dim L') = \min(j,\dim L''),
  \end{align}
  showing that \(H_{L}^{*}(X^{K})\) is a \(j\)-th syzygy over~\(R_{L}\).
\end{proof}

Combining Propositions~\ref{thm:depth-ABc}\,\ref{thm:depth-ABc-hommf}
and~\ref{thm:syzygy-depth-fixedpoints}\,\ref{thm:syzygy-depth-2},
we get:

\begin{corollary} $ $
  \label{thm:syzygy-ABc-fixedpoints}
  Assume that \(X\) is a \(T\)-manifold, and let \(j\ge0\).
  Then \(\HT^{*}(X)\) is a \(j\)-th syzygy if and only if
  \begin{equation*}
    H^{i}(\AB_{L,c}^{*}(X^{K};\kktilde))=0
  \end{equation*}
  for any isotropy subtorus~\(K\subset T\) occurring in~\(X\), its quotient~\(L=T/K\)
  and any \(i>\max(\dim L-j,0)\).
\end{corollary}

\section{Blowing up characteristic submanifolds}
\label{sec:blowup}

We assume for the rest of the paper that \(X\) is a \(T\)-manifold such that
\begin{enumerate}
\item \(X\) is non-empty and connected,
\item \(H^{*}(X)\) is finite-dimensional,
\item \(T\) acts effectively on~\(X\), and
\item all isotropy groups in~\(X\) are connected (hence tori).
\end{enumerate}
The second assumption is necessary in order to use the results of~\cite{AlldayFranzPuppe:orbits1}
and~\cite{AlldayFranzPuppe:orbits4}.
It implies that \(X\) has only finitely many components.
Since \(\HT^{*}(X)\) is a \(j\)-th syzygy if and only if this is true for each component
(see Section\nobreakspace \ref {sec:depth}),
it does not hurt to assume \(X\) to be connected. In the same vein we may appeal
to Lemma\nobreakspace \ref {thm:reduction-effective} and divide out a subtorus of~\(T\)
that acts trivially on~\(X\) to arrive at an almost effective action,
which will be effective by virtue of the fourth assumption.
As pointed out in the introduction, we do not expect the
latter assumption to be essential, either, but it ensures
that we need not consider orbifolds.

\begin{lemma}
  \label{thm:isotropy}
  Let \(X\) be a \(T\)-manifold (satisfying the assumptions above).
  \begin{enumroman}
  \item
    \label{thm:isotropy-1}
    Only finitely many isotropy groups occur in~\(X\).
  \item
    \label{thm:isotropy-2}
    \(X^{K}\) has only finitely many components for each isotropy subgroup~\(K\subset T\).
  \item
    \label{thm:isotropy-3}
    \(H^{*}(Y)\) is finite-dimensional for each such component~\(Y\).
  \end{enumroman}
\end{lemma}

\begin{proof}
  Because \(X\) is a manifold with finite-dimensional cohomology,
  only finitely many infinitesimal orbit types occur in~\(X\) \cite[Thm.~7.7]{Franz:nonab}.
  This proves the first claim since \(T\) is abelian and all isotropy groups are connected.
  The remaining parts are consequences
  of the localization theorem in equivariant cohomology, \cf~\cite[Cor.~3.10.2]{AlldayPuppe:1993}.
\end{proof}

A \emph{characteristic circle} for~\(X\) is a circle~\(K\subset T\)
that occurs as the isotropy group of some~\(x\in X\).
By Lemma\nobreakspace \ref {thm:isotropy}\,\ref{thm:isotropy-1},
there are only finitely many characteristic circles for~\(X\).
The finitely many connected components of~\(X^{K}\),
where \(K\) runs through the characteristic circles of~\(X\),
are called \emph{characteristic submanifolds} of~\(X\).
If they are all of codimension~\(2\),
we say that the \(T\)-action on~\(X\) is \emph{regular}.

We single out a special class  of regular actions.
For~\(n\ge 2r\)
consider the action of \((S^{1})^{r}\) on~\(\C^{r}\times\R^{n-2r}\) where
\((S^{1})^{r}\) acts in the usual way on~\(\C^{r}\) and
trivially on~\(\R^{n-2r}\).
We say that the \(T\)-action on~\(X\) is \emph{locally standard}
if \(n=\dim X\ge 2r\) and if each~\(x\in X\) has a \(T\)-stable open neighbourhood
weakly equivariantly diffeomorphic to an
\((S^{1})^{r}\)-stable neighbourhood of~\(\C^{r}\times\R^{n-2r}\),
that is,
equivariantly diffeomorphic with respect to some isomorphism~\(T\cong(S^{1})^{r}\).
It is clear that any locally standard action is regular.
The converse does not hold in general. For example, the
restriction of the usual action of~\((S^{1})^{2}\) on~\(\C^{2}\)
to the unit sphere~\(S^{3}\) is regular (with two circles as characteristic submanifolds),
but not locally standard.

\begin{lemma}
  \label{thm:regular-loc-std}
  Let \(X\) be a regular \(T\)-manifold such that \(X^{T}\ne\emptyset\).
  Then the action is locally standard.
\end{lemma}

\begin{proof}
  Assume first that \(X=V\) is an \(n\)-dimensional (real) representation of~\(T\),
  say with fixed subspace of dimension~\(m\)
  and non-zero weights~\(\chi_{1}\),~\dots,~\(\chi_{s}\), \(2s=n-m\).
  We claim that the~\(\chi_{j}\)'s form a basis of the character lattice of~\(T\).
  This will imply that \(V\) is weakly equivariantly diffeomorphic to
  \( 
    \C^{r}\times\R^{m}
  \). 

  If the weights do not span the character lattice, then there is
  a non-trivial subgroup~\(K\subset T\) acting trivially on~\(V\). 
  (If the weights span the lattice rationally, then \(K\) is discrete;
  otherwise it contains a circle.)
  But this is impossible since the action is effective.

  Now assume that the~\(\chi_{j}\)'s are linearly dependent.
  Then we must have \(s\ge2\), and at most~\(s-2\) of them,
  say, \(\chi_{1}\),~\ldots,~\(\chi_{s'}\), span a hyperplane~\(W\subset\ttt^{*}\)
  that does not contain \(\chi_{j}\) for~\(j>s'\).
  (Look at the reduced row echelon form of the matrix whose columns are the~\(\chi_{j}\)'s.)
  Consider a point~\(v\in V\) whose component corresponding to~\(\chi_{j}\)
  is non-zero for~\(j\le s'\) and zero otherwise.
  The circle~\(K\subset T\) whose Lie algebra is the annihilator of~\(W\)
  is the unique subcircle of~\(T\) such that \(v\in V^{K}\).
  Hence \(K\) is characteristic, but the codimension of~\(V^{K}\) in~\(V\) is \(2(s-s')\ge4\).
  This contradicts our assumption that \(T\) acts regularly.

  Now let \(X\) be arbitrary.
  For~\(x\in X\), let \(V\) be a differentiable slice at~\(x\) with isotropy subgroup~\(T_{x}\)
  of rank~\(k\).
  All isotropy groups of~\(T_{x}\) on~\(V\) are connected,
  and \(T_{x}\) acts effectively and regularly.
  By what we have just shown, this means that this action is locally standard.
  Hence an equivariant tubular neighbourhood of the orbit~\(Tx\) is
  weakly equivariantly diffeomorphic to the obvious action of~\((S^{1})^{r}\) on
  \begin{equation}
    \label{eq:locally}
    \C^{k} \times \R^{m} \times (S^{1})^{l}
  \end{equation}
  where \(l=r-k\).
  
  We have \(\dim X=2k+l+m\ge2r=2(k+l)\)
  because \(X\) is connected and the inequality holds near a fixed point.
  Hence \(m\ge l\), and we can replace \(\R^{l}\times(S^{1})^{l}\)
  by some suitable open subset of~\(\C^{l}\) in~\eqref{eq:locally}.
  This completes the proof.
\end{proof}

Let \(Y\) be a characteristic submanifold of~\(X\) with characteristic circle~\(K\subset T\),
and fix an isomorphism~\(\lambda\colon S^{1}\to K\).
Because all isotropy groups occurring in~\(X\) are connected,
the element~\(\lambda(e^{\pi i/2})\in K\) defines
a complex structure on the normal bundle~\(N_{Y}X\) of~\(Y\) in~\(X\)
such that scalar multiplication by elements of~\(S^{1}\subset\C\) coincides via~\(\lambda\) with the \(K\)-action.

This data is enough to define the
(complex projective) \emph{blow-up}~\(\tX\) of~\(X\) along~\(Y\) as a \(T\)-manifold,
\cf~\cite[pp.~602--605]{GriffithsHarris:1978}.\footnote{%
  In~\cite{GriffithsHarris:1978}, \(X\) and~\(Y\) are assumed to be complex analytic, so that
  the blow-up is again a complex manifold. This is not essential for the construction,
  nor for the validity of the short exact sequence~\eqref{eq:H-tX}.
  The \(T\)-action lifts by the naturality of the construction.}
The canonical \(T\)-equivariant projection
\begin{equation}
  \pi\colon \tX \to X
\end{equation}
restricts to a diffeomorphism~\(\tX\setminus \tY\to X\setminus Y\);
the preimage~\(\tY=\pi^{-1}(Y)\) is the projective bundle associated to~\(N_{Y}X\)
and of codimension~\(2\) in~\(\tX\).

\begin{lemma}
  \label{proj-blow}
  $ $
  \begin{enumroman}
  \item \label{proj-blow-1}
    \(\tX^{K}=\pi^{-1}(X^{K})\).
  \item \label{proj-blow-2}
    The characteristic circles of~\(\tX\) and~\(X\) are the same.
  \item \label{proj-blow-3}
    All isotropy groups in~\(\tX\) are connected.
  \item \label{proj-blow-4}
    \(H^{*}(\tX)\) is finite-dimensional.
  \end{enumroman}
\end{lemma}

As \(\tX\) is connected and \(T\) acts effectively on it,
this shows that \(\tX\) again satisfies the assumptions
on \(T\)-manifolds stated at the beginning of this section.

\begin{proof}
  For the first claim we observe
  that \(K\) acts trivially on~\(\tY\) and that \(\pi\colon\tX\setminus\tY\to X\setminus Y\)
  is an equivariant bijection.

  By the same token,
  new characteristic circles can only appear inside~\(\tY\). But we have just seen that all isotropy
  groups there contain~\(K\). Thus, no new characteristic circles occur. Now
  choose an~\(x\in Y\) such that \(T_{x}=K\), and let \(\tilde x\in\tY\) be a preimage of~\(x\).
  Then \(K\subset T_{\tilde x}\subset T_{x}=K\). Hence \(K\) occurs in~\(\tX\).
  This completes the proof of~\ref{proj-blow-2}.

  For~\ref{proj-blow-3} it is enough to look at some~\(\tilde x\in\tY\). Let \(x=\pi(\tilde x)\).
  Since \(Y\subset X^{K}\) is \(T_{x}\)-stable, so is its normal bundle in~\(X\) and in particular
  the fibre~\(V\) at~\(x\). By the discussion above, \(V\) is a complex \(T_{x}\)-module, and
  \(\pi^{-1}(x)\) can be identified with~\(\PP V\), so that \(\tilde x\) is the complex line
  through some~\(0\ne v\in V\). As \(v\) corresponds to some point in~\(U\), its isotropy group
  is connected.

  Consider the weight space decomposition~\(V=\bigoplus_{\gamma}V_{\gamma}\).
  An element of~\(T_{x}\) lies in~\(T_{\tilde x}\) if and only if it acts by the same scalar
  on all weight spaces for which the component of~\(v\) is non-zero, and it lies in~\(T_{v}\)
  if and only if this scalar equals \(1\).
  Since both~\(T_{v}\) and \(T_{\tilde x}/T_{v}\cong K\) are
  connected, so is \(T_{\tilde x}\).

  Because \(\pi\colon\tY\to Y\) is a projectivized complex vector bundle,
  \(H^{*}(\tY)\) is a finitely generated free module over~\(H^{*}(Y)\) via~\(\pi^{*}\),
  and the sequence
  \begin{equation}
    \label{eq:H-tX}
    0 \longrightarrow
    H^{*}(X) \longrightarrow
    H^{*}(\tX) \longrightarrow
    \frac{H^{*}(\tY)}{H^{*}(Y)} \longrightarrow
    0
  \end{equation}
  is exact, \cf~\cite[p.~605]{GriffithsHarris:1978}.
  Note that \(1\in H^{*}(\tY)\) can always be taken as part of a basis,
  so that the quotient~\(H^{*}(\tY)/H^{*}(Y)\) is again
  finitely generated free over~\(H^{*}(Y)\).
  It therefore follows from~\eqref{eq:H-tX} and Lemma\nobreakspace \ref {thm:isotropy}\,\ref{thm:isotropy-3}
  that \(H^{*}(\tX)\) is again finite-dimensional.
\end{proof}

If we blow up all components of all characteristic submanifolds of~\(X\),
we arrive at a connected \(T\)-manifold~\(\BlT X\), on which \(T\) acts regularly.
Moreover, if \(X^{T}\ne\emptyset\), then \(\BlT X\) is locally standard
by Lemma\nobreakspace \ref {thm:regular-loc-std}.
Note that \(\BlT X\) depends on the order of the various blow-ups
although we do not indicate that in our notation.

\smallskip

We now generalize the sequence~\eqref{eq:H-tX}
to equivariant cohomology and derive some consequences.

\begin{lemma}
  \label{thm:cohomology-blowup}
  Let \(X\) be a \(T\)-manifold, and let \(K\subset T\) be a characteristic circle.
  Let \(\pi\colon\tilde X\to X\) be the blow-up of~\(X\) along a component~\(Y\) of~\(X^{K}\),
  and set~\(\tY=\pi^{-1}(Y)\).
  Then \(\HT^{*}(\tY)\) is a free~\(\HT^{*}(Y)\)-module, and
  \begin{equation*}
    0 \longrightarrow
    \HT^{*}(X) \longrightarrow
    \HT^{*}(\tilde X) \longrightarrow
    \frac{\HT^{*}(\tY)}{\HT^{*}(Y)} \longrightarrow
    0
  \end{equation*}
  is exact.
\end{lemma}

\begin{proof}
  The map~\(ET\times_{T}\tY\to ET\times_{T}Y\) between the Borel constructions
  is a projectivized vector bundle, hence in cohomology
  the restriction map to the fibre is surjective.
  By the Leray--Hirsch theorem, \(\HT^{*}(\tY)\) is a free module over~\(\HT^{*}(Y)\).

  Let \(E\to B\) be a finite-dimensional approximation of~\(ET\to BT\). 
  Then the normal bundle of~\(E\times_{T}Y\) in~\(E\times_{T}X\) is~\(E\times_{T}N_{Y}X\),
  which implies that \(E\times_{T}\tX\) is the blow-up of~\(E\times_{T}X\) along~\(E\times_{T}Y\).
  Thus, the exactness of~\eqref{eq:H-tX} carries over.
\end{proof}

\begin{proposition}
  \label{thm:depth-syzygy-blowup}
  $ $
  \begin{enumroman}
  \item \label{thm:depth-blowup}
    \(\depth \HT^{*}(\BlT X)=\depth\HT^{*}(X)\).
  \item \label{thm:syzygy-blowup}
    For any~\(j\ge0\),
    \(\HT^{*}(\BlT X)\) is a \(j\)-th syzygy
    \(\:\Longleftrightarrow\:\)
    \(\HT^{*}(X)\) is a \(j\)-th syzygy.
  \end{enumroman}
\end{proposition}

Note that the second part is Proposition\nobreakspace \ref {thm:syzygy-blowup:intro}.

\begin{proof}
  It suffices to consider a single blow-up~\(\tilde X\to X\)
  along a component~\(Y\) of some~\(X^{K}\).
  If \(Y\) has codimension~\(2\) in~\(X\), then \(\tX=X\), and there is nothing to show.
  So assume that \(Y\) has a larger codimension.

  \ref{thm:depth-blowup}:
  Because \(\HT^{*}(\tY)\) is a free module over~\(\HT^{*}(Y)\),
  these two modules have the same depth over~\(R\), which
  is also the depth of the quotient~\(M=\HT^{*}(\tY)/\HT^{*}(Y)\).
  Using that \(\tY\) is a component of~\(\tX^{K}\) and \(\HT^{*}(\tY)\) therefore a direct summand of~\(\HT^{*}(\tX^{K})\),
  we get from
  Proposition\nobreakspace \ref {thm:syzygy-depth-fixedpoints}\,\ref{thm:syzygy-depth-1}
  that
  \begin{equation}
    \depth M = \depth \HT^{*}(\tY) \ge \depth \HT^{*}(\tX^{K}) \ge \depth \HT^{*}(\tX).
  \end{equation}
  From Lemma\nobreakspace \ref {thm:cohomology-blowup} and
  the behaviour of depth in short exact sequences,
  \cf~\cite[Prop.~1.2.9]{BrunsHerzog:1998},
  we conclude
  \begin{align}
    \depth \HT^{*}(\tX) &\ge \depth\HT^{*}(X)
    \shortintertext{and}
    \depth\HT^{*}(X) &\ge \min\bigr(\depth \HT^{*}(\tX),\depth M+1\bigl) \ge \depth \HT^{*}(\tX),
  \end{align}
  as desired.

  \ref{thm:syzygy-blowup}:
  This part is similar.
  If \(\HT^{*}(X)\) is a \(j\)-th syzygy over~\(R\), then so is \(\HT^{*}(Y)\)
  by Proposition\nobreakspace \ref {thm:syzygy-depth-fixedpoints}\,\ref{thm:syzygy-depth-2}
  (or part~\ref{thm:syzygy-XK} together with Lemma\nobreakspace \ref {thm:reduction-effective}\,\ref{thm:reduction-effective-HT}).
  In the same vein, if \(\HT^{*}(\tX)\) is a \(j\)-th syzygy, then so is \(\HT^{*}(\tY)\),
  hence also \(\HT^{*}(Y)\). So in any case we can assume that \(M\) is a \(j\)-th syzygy.
  Now we localize the short exact sequence from Lemma\nobreakspace \ref {thm:cohomology-blowup} at~\(\pp\lhd R\),
  \begin{equation}
    0 \longrightarrow
    \HT^{*}(X)_{\pp} \longrightarrow
    \HT^{*}(\tX)_{\pp} \longrightarrow
    M_{\pp} \longrightarrow
    0.
  \end{equation}
  From the characterization~\eqref{eq:condition-syzygy} of syzygies in terms of depth
  we get that the depth of \(\HT^{*}(X)_{\pp}\) or~\(\HT^{*}(\tX)_{\pp}\)
  is bounded below by~\(\min(j,\depth R_{\pp})\), and the same applies to~\(M_{\pp}\).
  Reasoning in the same way as above, we see that this bound holds
  for the remaining term as well.
  We conclude by applying \eqref{eq:condition-syzygy} a second time.
\end{proof}

\section{A decomposition of the Atiyah--Bredon sequence}
\label{sec:decomposition}

Let \(X\) be a \(T\)-manifold of dimension~\(n\).
In addition to the blanket assumptions stated in Section\nobreakspace \ref {sec:blowup},
we require in this section that the \(T\)-action is \emph{locally standard},
that is, locally weakly equivariantly diffeomorphic to
\( 
  \C^{r}\times\R^{n-2r}
\). 
It then follows for example from a result of Schwarz~\cite[Thm.~1]{Schwarz:1975}
that the quotient is locally diffeomorphic to
\( 
  [0,\infty)^{r}\times\R^{n-2r}
\). 
This shows that \(X/T\) is a manifold with corners
and in fact a nice manifold with corners, \cf~\cite[Sec.~5.1]{MasudaPanov:2006}.
Here ``nice'' means that each codimension-\(k\) face is contained in
exactly \(k\) facets (codimension-\(1\) faces), which is a consequence
of the orbit structure.
(Note that we assume all faces to be closed and connected.)

The main result of this section (Theorem\nobreakspace \ref {thm:iso-HABc-HBc})
relates the cohomology of the non-augmented Atiyah--Bredon complex~\(\ABc^{*}(X\AAopt)\)
with constant (\(\AA=\kk\)) or twisted (\(\AA=\kktilde\)) coefficients
to the cohomology of the faces of~\(X/T\).
This description will be fundamental to all our applications in Section\nobreakspace \ref {sec:previous-work}.

Let \(\pi\colon X\to X/T\) be the quotient map;
we sometimes identify \(X^{T}\) with its image in~\(X/T\).
Let \(P\subset X/T\) be a face. 
We write \(\dot P\) for the relative interior of~\(P\) and
\(\partial P=P\setminus\dot P\) for its boundary.
We will often use that \(P\) is a topological manifold with boundary~\(\partial P\).
The inclusion~\(\dot P\hookrightarrow P\) induces an isomorphism
\begin{equation}
  \label{eq:P-incl-int}
  H_{*}(\dot P) = H_{*}(P) \,;
\end{equation}
\cf~\cite[Lemma~11.7]{Massey:1978}.
We define \(T_{P}\) to be the common isotropy group of the points in~\(X\) lying over~\(\dot P\)
and set \(T^{P}=T/T_{P}\).
The codimension of~\(P\) in~\(X/T\) is the rank of~\(T_{P}\).

The faces of~\(X/T\) form a poset~\(\PPP\) under inclusion.
The \emph{rank} of a face~\(P\in\PPP\) is the rank of~\(T^{P}\).
We write \(P\facet Q\) if \(P\) is a facet of~\(Q\),
that is, if \(P<Q\) and \(\rank P=\rank Q-1\).

The preimage~\(X^{P}=\pi^{-1}(P)\) of a face~\(P\in\PPP\)
is a connected component of~\(X^{T_{P}}\) and in particular
a \(T\)-stable closed submanifold of~\(X\) of codimension~\(2\codim P\).
The characteristic submanifolds of~\(X\) are exactly the~\(X^{F}\) with~\(F\facet X/T\),
and likewise the characteristic circles the corresponding~\(T_{F}\)'s.
The preimage
\begin{equation}
  \piinvdot{P}=\pi^{-1}(\dot P) 
\end{equation}
is the open subset of~\(X^{P}\) on which \(T^{P}\) acts freely.

For~\(F\facet X/T\), we write \(N^{F}\) for the normal bundle of~\(X^{F}\) in~\(X\),
and we fix an isomorphism~\(\lambda_{F}\colon S^{1}\to T_{F}\).
As in Section\nobreakspace \ref {sec:blowup}, this defines a complex structure on~\(N^{F}\)
and hence an orientation.
If \(P\le F\facet X/T\), we write
\begin{equation}
  e_{F}^{P}\in\HT^{2}(\piinvdot{P})
\end{equation}
for the restriction of the equivariant Euler class of~\(N^{F}\) to~\(\piinvdot{P}\).

Since the \(T\)-action is locally standard, the following is immediate:

\begin{lemma}
  \label{thm:intersection-XF}
  Let \(P\in\PPP\). Then \(X^{P}\) is a connected component
  of the intersection of the~\(X^{F}\) with~\(P\le F\facet X/T\).
  Moreover, this intersection is transversal, and \(T_{P}\) is the product of the
  corresponding characteristic circles~\(T_{F}\).
\end{lemma}

We also define the algebra~\(R_{P}=H^{*}(B T_{P})\)
as well as the canonical restriction maps~\(\rho_{P}\colon R\to R_{P}\)
and~\(\rho_{P Q}\colon R_{P}\to R_{Q}\) for~\(P\le Q\).
For~\(P\le F\facet X/T\)
we let \(t_{F}^{P}\in R_{P}\) be the restriction of~\(e_{F}^{P}\)
to any point in~\(X^{P}\); it is the weight of the \(T_{P}\)-action on the
restriction of~\(N^{F}\) to~\(X^{P}\). Note that for~\(P\le Q\) we have
\begin{equation}
  \label{eq:canonical-generators}
  \rho_{P Q}\bigl(t_{F}^{P}\bigr) =
  \begin{cases}
    t_{F}^{Q} & \text{if \(Q\le F\),} \\
    0 & \text{otherwise}
  \end{cases}
\end{equation}
by construction.

For any~\(P\in\PPP\) the elements~\(\tP_{F}\) with~\(P\le F\facet X/T\)
form a basis of the vector space~\(H^{2}(BT_{P})\) by Lemma\nobreakspace \ref {thm:intersection-XF},
whence a canonical isomorphism of algebras
\begin{equation}
  \label{eq:iso-RP-RtF}
  R_{P} \cong \kk\bigl[\tP_{F}:P\le F\facet X/T\bigr],
  \qquad
  t \mapsto \!\! \sum_{P\le F\facet X/T} \!\! \pair{t,\lambda_{F}}\, \tP_{F},
\end{equation}
where we consider \(\lambda_{F}\) as an element of~\(H_{1}(T_{P})\)
and write \(\pair{-,-}\) for the canonical pairing between~\(H_{1}(T_{P})\)
and \(H^{1}(T_{P})=H^{2}(BT_{P})\).

The \(T\)-action on~\(X\) lifts to the orientation cover~\(\tilde X\), \cf~\citeorbitsfour{Lemma~2.13}.
Hence \(\tilde X/T\to X/T\) is a two-fold cover.
Whenever we use twisted coefficients for~\(X/T\) or subsets thereof, we refer to this two-fold cover.
Because it restricts to the orientation cover over each open face~\(\dot P\),
we recover the usual meaning of (co)homology with twisted coefficients in this case.
In particular, the preimage~\(\tilde P\) of each face~\(P\) in~\(\tilde X/T\)
is an orientable topological manifold with boundary in the sense of~\cite[\S 11.2]{Massey:1978}.

Let \(Y\) be a closed subset of some space~\(Z\). Recall that for cohomology with compact supports,
the role of the relative cohomology of the pair~\((Z,Y)\) is played by the cohomology of the complement~\(Z\setminus Y\),
\cf~\cite[Sec.~1.4]{Massey:1978}.
In particular, the connecting homomorphism in the long exact cohomology sequence 
is a map
\begin{equation}
  \delta=\delta_{Z,Y}\colon\Hc^{*}(Y)\to\Hc^{*+1}(Z\setminus Y).
\end{equation}
The same remark applies to equivariant cohomology and to twisted coefficients.

\begin{lemma}
  \label{thm:euler-vanish}
  Let \(Y\) be a closed \(T\)-stable submanifold of some \(T\)-manifold~\(Z\), and let \(e_{Y}\)
  be the equivariant Euler class of its normal bundle. Then
  \begin{equation*}
    \delta(\alpha\,e_{Y}) = 0 \in \HTc^{*}(Z\setminus Y\AAopt)
  \end{equation*}
  for any~\(\alpha\in\HTc^{*}(Y\AAopt)\).
\end{lemma}

\begin{proof}
  Let \(E\) be a \(T\)-stable tubular neighbourhood of~\(Y\) in~\(Z\).
  The commutative diagram
  \begin{equation}
    \begin{tikzcd}
      \HTc^{*}(E\AAopt) \arrow{r} \arrow{d} & \HTc^{*}(Y\AAopt) \arrow{r}{\delta} \arrow{d}{=} & \HTc^{*+1}(E\setminus Y\AAopt) \arrow{d} \\
      \HTc^{*}(Z\AAopt) \arrow{r} & \HTc^{*}(Y\AAopt) \arrow{r}{\delta} & \HTc^{*+1}(Z\setminus Y\AAopt)
    \end{tikzcd}
  \end{equation}
  shows that \(\delta\) vanishes on any~\(\beta\in\HTc^{*}(Y\AAopt)\) in the image of the restriction map
  from~\(E\) to~\(Y\). By the Thom isomorphism, these are exactly the multiples of~\(e_{Y}\).
\end{proof}

For \(P\in\PPP\) we consider the composition
\begin{equation}
  \label{eq:def-psi-P}
  \phi_{P}\colon \Hc^{*}(\dot P\AAopt) 
  \stackrel{\cong}{\longrightarrow}
  H_{\TPc}^{*}(\piinvdot{P}\AAopt) \longrightarrow \HTc^{*}(\piinvdot{P}\AAopt) \,;
\end{equation}
the first map is an isomorphism because \(T^{P}\) acts freely on~\(\piinvdot{P}\)
with quotient~\(\dot P\). Taking \eqref{eq:iso-RP-RtF} into account, we define
\begin{align}
  \psi_{P}\colon \Hc^{*}(\dot P\AAopt)\otimes R_{P} &\to \HTc^{*}(\piinvdot{P}\AAopt), \\
  \alpha\, \tP_{F_{1}}\cdots \tP_{F_{k}} &\mapsto
  \phi_{P}(\alpha)\,e_{F_{1}}^{P}\cdots e_{F_{k}}^{P}.
\end{align}
(The \(F_{1}\),~\dots,~\(F_{k}\) above need not be distinct, and the ``\(\otimes\)'' on the left is suppressed.)

\begin{lemma}
  \label{thm:psi-P-iso}
  Let \(P\in\PPP\).
  \begin{enumroman}
  \item The map~\(\psi_{P}\) is an isomorphism of graded vector spaces.
  \item For \(P\facet Q\) the following diagram commutes:
    \begin{equation*}
      \begin{tikzcd}
	\Hc^{*}(\dot P\AAopt)\otimes R_{P} \arrow{r}{\psi_{P}} \arrow{d}[left]{\delta\otimes\rho_{PQ}} & \HTc^{*}(\piinvdot{P}\AAopt) \arrow{d}{\delta} \\
 	\Hc^{*+1}(\dot Q\AAopt)\otimes R_{Q} \arrow{r}{\psi_{Q}} & \HTc^{*+1}(\piinvdot{Q}\AAopt)
      \end{tikzcd}
    \end{equation*}
  \end{enumroman}
\end{lemma}

Note that the maps~\(\delta\) in the diagram above refer to the connecting homomorphisms
for the pair~\((\dot Q\cup\dot P,\dot P)\) and its preimage under~\(\pi\).

\begin{proof}
  \def\kkcp{\kappa_{c}}
  \def\kkcl{\kappa}
  Given the face~\(P\),
  we choose a splitting~\(T\cong T^{P}\times T_{P}\). 
  By Lemma\nobreakspace \ref {thm:reduction-effective}\,\ref{thm:reduction-effective-HT}
  this induces an isomorphism of graded vector spaces
  \begin{equation}
    \label{eq:def-kappa-P-c}
    \kkcp\colon
    \Hc^{*}(\dot P\AAopt)\otimes R_{P} \to \HTc^{*}(\piinvdot{P}\AAopt)
  \end{equation}
  defined similarly to~\(\psi_{P}\); only the multiplication by Euler classes is replaced
  by the restriction of the \(R\)-module structure to~\(R_{P}\) given by the splitting of~\(T\).
  (See~\cite[p.~1351 \& Rem.~2.15]{AlldayFranzPuppe:orbits4} or \cite[Prop.~7.1]{Franz:nonab}
  for the extension of Lemma\nobreakspace \ref {thm:reduction-effective}\,\ref{thm:reduction-effective-HT}
  to compact supports and twisted coefficients.)
  To prove our first claim,
  it suffices to show that \(\kkcp^{-1}\circ\psi_{P}\)
  is a vector space automorphism of~\(\Hc^{*}(\dot P\AAopt)\otimes R_{P}\).

  We start by looking more closely at the equivariant Euler classes.
  For cohomology with closed supports there is an isomorphism of algebras
  \begin{equation}
    \kkcl\colon H^{*}(\dot P)\otimes R_{P} \to \HT^{*}(\piinvdot{P})
  \end{equation}
  analogous to~\eqref{eq:def-kappa-P-c},
  and \(\kkcp\) is equivariant with respect to~\(\kkcl\)
  in the sense that
  \begin{equation}
     \kkcp(\alpha\,\tilde\alpha) = \kkcp(\alpha)\,\kkcl(\tilde\alpha)
  \end{equation}
  for~\(\alpha\in\Hc^{*}(\dot P\AAopt)\otimes R_{P}\) and~\(\tilde\alpha\in H^{*}(\dot P)\otimes R_{P}\).
  (Recall that \(\Hc^{*}(\dot P\AAopt)\) is a module over~\(H^{*}(\dot P)\),
  and likewise for equivariant cohomology.)
  For~\(\beta\in \HT^{*}(\piinvdot{P})\),
  the component of~\(\kkcl^{-1}(\beta)\) in~\(H^{0}(\dot P)\otimes R_{P} = R_{P}\)
  is the restriction of~\(\beta\) to any point~\(x\in\piinvdot{P}\).
  For the equivariant Euler class~\(e_{F}^{P}\)
  this is the weight~\(\tP_{F}\), as remarked above.
  Thus,
  \begin{equation}
    \kkcl^{-1}(e_{F}^{P}) = \tP_{F}+\gamma_{F}
  \end{equation}
  for some~\(\gamma_{F}\in H^{2}(\dot P)\).

  By what we have said so far, we have for~\(\alpha\in\Hc^{*}(\dot P\AAopt)\)
  \begin{align}
    (\kkcp^{-1}\circ\psi_{P})(\alpha\, \tP_{F_{1}}\cdots \tP_{F_{k}})
    &= \alpha \prod_{i=1}^{k} \kkcl^{-1}(e_{F_{i}}^{P}) =
    \alpha \prod_{i=1}^{k} \bigl(\tP_{F_{i}}+\gamma_{F_{i}}) \\
    &= \alpha\, \tP_{F_{1}}\cdots \tP_{F_{k}} + \text{terms of lower degree in~\(R_{P}\)}.
  \end{align}
  This implies that \(\kkcp^{-1}\circ\psi_{P}\) is bijective.
  
  We now turn to the second claim.
  Because the maps~\(\phi_{P}\) commute with the coboundary maps~\(\delta\)
  and \(R_{P}\) is generated by the~\(\tP_{F}\) with~\(P\le F\),
  it suffices to verify that for~\(\beta\in\HTc^{*}(\piinvdot{P}\AAopt)\) one has
  \begin{equation}
    \delta(\beta\,e_{F}^{P}) =
    \begin{cases}
      \delta(\beta)\,e_{F}^{Q} & \text{if \(Q\le F\),} \\
      0 & \text{otherwise.}
    \end{cases}
  \end{equation}
  As both \(e_{F}^{P}\) and \(e_{F}^{Q}\) are restrictions of the equivariant Euler class of~\(N^{F}\),
  the first case is a consequence of the linearity of~\(\delta\) over~\(\HT^{*}(X^{Q})\).
  The second case follows from Lemma\nobreakspace \ref {thm:euler-vanish}
  because the normal bundle of~\(X^{P}\) in~\(X^{Q}\) is the restriction
  of~\(N^{F}\). 
\end{proof}

The filtration of a face~\(P\) of~\(X/T\) by its faces
leads to a spectral sequence for cohomology with compact supports and coefficients~\(\AA\).
The \(i\)-th column of the \(E_{1}\)~page is
\begin{equation}
  \label{eq:ss-P-E1}
  \Bc^{i}(P\AAopt) \vcentcolon=
  \!\! \bigoplus_{\substack{Q\le P \\ \rank Q=i}} \!\! \Hc^{*+i}(Q,\partial Q\AAopt)
  = \!\! \bigoplus_{\substack{Q\le P \\ \rank Q=i}} \!\! \Hc^{*+i}(\dot Q\AAopt)
\end{equation}
with differential
\begin{equation}
  d\alpha = \bigoplus_{Q\facet O} \delta_{\dot O\cup\dot Q,\dot Q}\,\alpha
\end{equation}
for~\(\alpha\in\Hc^{k}(Q,\partial Q\AAopt)\).
Note that the degree shifts in~\eqref{eq:ss-P-E1} ensure that \(\alpha\) 
has total degree~\(k\) in~\(\Bc^{*}(P\AAopt)\), independently of the rank of~\(Q\).

Let us also introduce the monomials
\begin{equation}
  t_{P} = \prod_{P\le F}\tP_{F} \in R_{P}
\end{equation}
for~\(P\in\PPP\). These elements give rise to an isomorphisms of graded vector spaces
\begin{equation}
  \label{eq:iso-RP-sum-RQ}
  R_{P} = \bigoplus_{P\le Q} R_{Q}\, t_{Q}
\end{equation}
by collecting monomials in same generators.
For example, one can decompose \(R=\kk[t_{1},t_{2}]\) as
\begin{equation}
  \kk[t_{1},t_{2}] = \kk \oplus \kk[t_{1}]\,t_{1} \oplus \kk[t_{2}]\,t_{2} \oplus \kk[t_{1},t_{2}]\,t_{1}t_{2}.
\end{equation}

The following result is central to our work.

\begin{theorem}
  \label{thm:iso-HABc-HBc}
  For any~\(i\ge0\) there is an isomorphism of graded vector spaces
  \begin{equation*}
    H^{i}(\ABc^{*}(X\AAopt)) =
    \bigoplus_{P\in\PPP} H^{i}(\Bc^{*}(P\AAopt))\otimes R_{P}\, t_{P}.
  \end{equation*}
\end{theorem}

\begin{proof}
  By Lemma\nobreakspace \ref {thm:psi-P-iso},
  the isomorphisms~\(\psi_{P}\) combine for any~\(i\ge0\) to an isomorphism of graded vector spaces
  \begin{equation}
    \!\! \bigoplus_{\rank P=i} \!\! \Hc^{*+i}(\dot P\AAopt)\otimes R_{P}
    \to
    \!\! \bigoplus_{\rank P=i} \!\! \HTc^{*+i}(\piinvdot{P}\AAopt)
    = \ABc^{i}(X\AAopt) \,;
  \end{equation}
  these isomorphisms are compatible with the differentials.
  On the other hand, using \eqref{eq:iso-RP-sum-RQ} we can rewrite the left-hand side as
  \begin{align}
    \!\! \bigoplus_{\rank P=i} \!\! \Hc^{*+i}(\dot P\AAopt)\otimes R_{P}
    &= \!\! \bigoplus_{\substack{P\le Q \\ \rank P=i}} \!\! \Hc^{*+i}(\dot P\AAopt)\otimes R_{Q}\, t_{Q} \\
    &= \bigoplus_{Q\in\PPP} \Bc^{i}(Q\AAopt)\otimes R_{Q}\, t_{Q}.
  \end{align}
  The claim now follows by taking cohomology.
\end{proof}

\begin{theorem}
  \label{thm:condition-syzygy-HBc}
  Let \(0\le j\le r\). Then \(\HT^{*}(X)\) is a \(j\)-th syzygy if and only if
  \(H^{i}(\Bc^{*}(P;\kktilde))=0\) for any~\(P\in\PPP\) and any~\(i>\max(\rank P-j,0)\).
\end{theorem}

\begin{proof}
  By Corollary\nobreakspace \ref {thm:syzygy-ABc-fixedpoints}, \(\HT^{*}(X)\)
  is a \(j\)-th syzygy if and only if
  \begin{equation}
    \label{eq:AB-vanishes}
    H^{i}(\AB_{\TPc}^{*}(X^{P};\kktilde)) = 0
  \end{equation}
  for all~\(P\in\PPP\) and all~\(i>\max(\rank P-j,0)\).
  Theorem\nobreakspace \ref {thm:iso-HABc-HBc} implies that \eqref{eq:AB-vanishes}
  is equivalent to the vanishing of~\(H^{i}(\Bc^{*}(Q;\kktilde))\)
  for all~\(Q\le P\). Taking the latter condition for all~\(P\) and~\(i\) as before gives the result.
\end{proof}

Next we introduce the complexes~\(\BB^{*}(P)\) mentioned in the introduction.
While they turn out to be isomorphic to the~\(\Bc^{*}(P;\kktilde)\),
they are sometimes more convenient in applications, for example because
they do not involve twisted coefficients.

The following observation is well-known
for orientable topological manifolds with boundary,
see~\cite[p.~348]{Massey:1978}. The non-orientable case
follows as usual by considering the \((-1)\)-eigenspaces of the (co)homological maps
induced by the non-trivial deck transformation of the orientable two-fold cover.

\begin{lemma}
  \label{thm:poincare-boundary}
  Let \(P\in\PPP\) be of rank~\(p\), hence of dimension~\(m=n-2r+p\), and let \(Q\facet P\).
  The image of an orientation~\(o_{P}\) of~\(\dot P\)
  under the boundary map~\(\partial\colon\hHc_{m}(\dot P;\kktilde)\to\hHc_{m-1}(\dot Q;\kktilde)\)
  is an orientation of~\(\dot Q\).
  Moreover, the diagram
  \begin{equation*}
    \begin{tikzcd}[row sep=large]
      \Hc^{*-1}(\dot Q;\kktilde) \arrow{r}{\delta} \arrow{d}[left]{\partial o_{P}\cap{}} &
      \Hc^{*}(\dot P;\kktilde) \arrow{d}[left]{o_{P}\cap{}} \\
      H_{m-*}(Q) \arrow{r}{(\iota_{QP})_{*}} & H_{m-*}(P) 
    \end{tikzcd}
  \end{equation*}
  commutes, and the vertical arrows are isomorphisms.
\end{lemma}

Here \(\iota_{QP}\) denotes the inclusion~\(Q\hookrightarrow P\).
Also, we consider the cap product with~\(o_{P}\) as the composition
of the cap product~\(\Hc^{*}(\dot P;\kktilde)\to H_{m-*}(\dot P)\)
and the isomorphism~\eqref{eq:P-incl-int}.
The same applies to the cap product with~\(\partial o_{P}\).

Let \(P\in\PPP\) be of rank~\(p\), and let \(F_{1}\),~\dots,~\(F_{s}\) be the facets of~\(P\).
If \(Q\le P\) is a face of rank~\(q\),
then \(Q\) is a connected component of the intersection of \(p-q\) facets
\(F_{i_{1}}\),~\dots,~\(F_{i_{p-q}}\) of~\(P\) with~\(i_{1}<\dots<i_{p-q}\).
We write \(F_{i_{1}\dots i_{p-q}}\) for this latter intersection, and
\(Q_{(k)}\) for the connected component of \(F_{i_{1}\dots i_{k-1}i_{k+1}\dots i_{p-q}}\) 
containing \(Q\).
Then \(Q_{(1)}\),~\ldots,~\(Q_{(p-q)}\) are exactly the faces of~\(P\)
that have \(Q\) as a facet.

We define 
a complex~\(\BB^{*}(P)\) of graded vector spaces by
\begin{align}
  \BB^{q}(P) = \!\! \bigoplus_{\substack{Q\le P \\ \rank Q=q}} \!\! H_{*}(Q)
  \intertext{with differential}
  \label{eq:definition-d-BB}
  d c = \sum_{k=1}^{p-q} \! (-1)^{k-1}(\iota_{Q Q_{(k)}})_{*}(c)
\end{align}
for~\(c\in H_{i}(Q)\subset\BB^{q}(P)\),
where the \(Q_{(k)}\)'s are defined as above.
It will be convenient for us to grade homology negatively in the sense
that we define the total degree of~\(c\) in~\(\BB^{*}(P)\) to be~\(q-i\).
It is not difficult to check directly that \eqref{eq:definition-d-BB} indeed defines a differential;
it will also be implicit in the proof of the following observation.

\begin{lemma}
  \label{thm:Bc-BB-iso}
  Let \(P\in\PPP\).
  The complexes~\(\Bc^{*}(P;\kktilde)\) and \(\BB^{*}(P)\) are isomorphic
  up to a degree shift by~\(2r-n\).
\end{lemma}

\begin{proof}
  Set \(p=\rank P\) and \(m=n-2r\).
  We continue to use the notation introduced above and
  fix an orientation \(o_{P}\in\hHc_{m+p}(\dot P;\kktilde)\) of~\(P\).
  For each~\(Q\le P\), say of rank~\(q\) and a connected component of~\(F_{i_{1}\dots i_{p-q}}\),
  we define the orientation \(o_{Q}\in\hHc_{m+q}(\dot Q;\kktilde)\)
  to be the image of~\(o_{P}\) under the iterated boundary map
  \begin{multline}
    \label{eq:def-orientation-Q}
    \hHc_{m+p}(\dot P;\kktilde) \stackrel{\partial}\longrightarrow
    \hHc_{m+p-1}(\dot F_{i_{p-q}};\kktilde) \stackrel{\partial}\longrightarrow
    \hHc_{m+p-2}(\dot F_{i_{p-q-1}i_{p-q}};\kktilde) \stackrel{\partial}\longrightarrow \cdots \\
    \stackrel{\partial}\longrightarrow
    \hHc_{m+q+1}(\dot F_{i_{2}\dots i_{p-q}};\kktilde) \stackrel{\partial}\longrightarrow
    \hHc_{m+q}(\dot F_{i_{1}\dots i_{p-q}};\kktilde) \longrightarrow
    \hHc_{m+q}(\dot Q;\kktilde).
  \end{multline}
  Note that \(o_{Q}=\partial o_{Q_{(1)}}\) by definition.
  More generally,
  \begin{equation}
    \label{eq:oQ-oQ1}
    o_{Q}=(-1)^{k-1}\partial o_{Q_{(k)}}
  \end{equation}
  because a permutation of the facets~\(F_{i_{k}}\)
  in formula~\eqref{eq:def-orientation-Q}
  changes \(o_{Q}\) by the sign of the permutation.
  (Since this is true for~\([0,\infty)^{p-q}\) and
  \(P\) looks like \([0,\infty)^{p-q}\times Q\) locally around~\(Q\), it holds in general
  by 
  naturality.)

  We define a bijection~\(\Bc^{*}(P;\kktilde)\to\BB^{*}(P)\)
  by sending \(\alpha\in\Hc^{k}(\dot Q;\kktilde)\)
  to~\(\alpha\cap o_{Q}\in H_{m+q-k}(Q)\).
  This map has degree~\(-m\) since
  \(\alpha\) has total degree~\(k\) in~\(\Bc^{*}(P;\kktilde)\) and
  \(\alpha\cap o_{Q}\) total degree~\(k-m\) in~\(\BB^{*}(P)\).
  That it is compatible with the differentials
  follows from Lemma\nobreakspace \ref {thm:poincare-boundary} and identity~\eqref{eq:oQ-oQ1}.
  (The isomorphism of bigraded vector spaces
  established this way also shows \(\BB^{*}(P)\)
  is actually a complex with the given differential.)
\end{proof}

Combining Lemma\nobreakspace \ref {thm:Bc-BB-iso} with Theorem\nobreakspace \ref {thm:condition-syzygy-HBc},
we finally get Theorem\nobreakspace \ref {thm:condition-syzygy-HBc-intro}:

\begin{theorem}
  \label{thm:condition-syzygy-HBB}
  Let \(0\le j\le r\). Then
  \(\HT^{*}(X)\) is a \(j\)-th syzygy if and only if
  \(H^{i}(\BB^{*}(P))=0\)
  for all~\(P\in\PPP\) and all~\(i>\max(\rank P-j,0)\).
\end{theorem}

\section{Multiplicative structure}
\label{sec:product}

We still assume that \(X\) is a manifold with a locally standard \(T\)-action,
and for simplicity we additionally assume it to be compact,
so that the fixed point set~\(X_{0}\) is also compact.

Then \(\AB^{0}(X)=\HT^{*}(X_{0})\) has a product. Because \(H^{0}(\AB^{*}(X))\)
is contained in~\(\AB^{0}(X)\), we can consider the product of two elements
of~\(H^{0}(\AB^{*}(X))\) as an element of~\(\AB^{0}(X)\). Note that the image
of~\(\HT^{*}(X)\) in~\(H^{0}(\AB^{*}(X))\) is multiplicatively closed. In particular,
if \(\HT^{*}(X)\) is reflexive, then \(H^{0}(\AB^{*}(X))=\HT^{*}(X)\) is an algebra.

Similarly, for any face~\(P\in\PPP\)
we may consider the product of two elements of~\(H^{0}(\B^{*}(P))\)
as an element of~\(\B^{0}(P)\). Moreover, for any face~\(Q\le P\)
there is a restriction map~\(H^{0}(\B^{*}(P))\to H^{0}(\B^{*}(Q))\).
Given \(\alpha\in H^{0}(\B^{*}(P))\), we write \(\at{\alpha}_{Q}\)
for its restriction to~\(Q\).
We also write elements of~\(H^{0}(\B^{*}(P))\otimes R_{P}t_{P}\)
in the form~\(\alpha f t_{P}\) with~\(\alpha\in H^{0}(\B^{*}(P))\) and \(f\in R_{P}\).

\begin{proposition}
  \label{thm:product-AB0}
  The product of~\(\alpha f t_{P}\in H^{0}(\B^{*}(P))\otimes R_{P}t_{P}\subset H^{0}(\AB^{*}(X))\)
  and~\(\beta g t_{Q}\in H^{0}(\B^{*}(Q))\otimes R_{Q}t_{Q}\subset H^{0}(\AB^{*}(X))\)
  is given by
  \begin{equation*}
    \alpha \, f \, t_{P} \cdot \beta \, g \, t_{Q} =
    \sum_{O\subset P\cap Q}\at{\alpha}_{O}\at{\beta}_{O} (f g \, t_{P\vee Q})\,t_{O}
    \in \AB^{0}(X).
  \end{equation*}
\end{proposition}

Here \(P\vee Q\) denotes the smallest face of~\(X/T\) containing both~\(P\) and~\(Q\),
and the sum extends over all components~\(O\in\PPP\) of the intersection~\(P\cap Q\).
For each pair of faces~\(N\le O\) we think of~\(R_{O}\) as a subalgebra of~\(R_{N}\)
by identifying \(t^{O}_{F}\) with~\(t^{N}_{F}\). We can then say that \(R_{O}\subset R_{N}\) is
generated by the~\(t_{F}\in R_{N}\) with~\(O\le F\).

\begin{proof}
  We consider \(\alpha f t_{P}\), \(\beta g t_{Q}\) and their product as elements of
  \begin{equation}
    \AB^{0}(X) = \HT^{*}(X_{0})
    = \bigoplus_{\rank N = 0} \Hc^{*}(N)\otimes R
  \end{equation}
  with componentwise product. The copy of~\(R\) indexed by~\(N\)
  decomposes as in~\eqref{eq:iso-RP-sum-RQ} according to the faces containing \(N\).

  The restriction of~\(\alpha f t_{P}\) to~\(N\) is \(\at{\alpha}_{N} f t_{P}\) if \(N\le P\)
  and \(0\) otherwise; similarly for~\(\beta\) and the claimed product.
  Because multiplication is componentwise, we only have to consider 
  the rank-\(0\) faces~\(N\) contained in~\(P\cap Q\).
  For such an~\(N\), say contained in the component~\(O\) of~\(P\cap Q\), we have
  \begin{equation}
    t_{P}\,t_{Q} =  t_{P\vee Q}\,t_{O} \in R_{N} = R,
  \end{equation}
  which proves the claim.
\end{proof}

We compare \(H^{0}(\AB^{*}(X))\) to the face ring of the poset~\(\PPP\),
which is defined as
\begin{equation}
  \label{eq:face-ring-poset}
  \kk[\PPP] = \kk\bigl[t_{P}:P\in\PPP\bigr] \bigm/
  \bigl( t_{X/T}-1, t_{P}\,t_{Q}-t_{P\vee Q}\!\sum_{O\in P\cap Q}t_{O}: P,Q\in\PPP\bigr),
\end{equation}
\cf~\cite[Sec.~5]{MasudaPanov:2006}.
The face ring is an \(R\)-algebra via the morphism of algebras
\begin{equation}
  \label{eq:face-ring-R-module}
  R\mapsto \kk[\PPP],
  \quad
  t \mapsto \!\! \sum_{\text{\(F\)~facet}} \!\! \pair{t,\lambda_{F}}\,t_{F},
\end{equation}
where we have used the same notation as in~\eqref{eq:iso-RP-RtF}.

\begin{proposition}
  \label{thm:H0AB-face-ring}
  $ $
  \begin{enumroman}
  \item \label{pr1}
    If \(H^{0}(\B^{*}(P))=\kk\) for every~\(P\in\PPP\),
    then there is an isomorphism of \(R\)-algebras
    \begin{equation*}
      H^{0}(\AB^{*}(X)) \cong \kk[\PPP].
    \end{equation*}
  \item \label{pr2}
    If the condition in~\ref{pr1} is satisfied for all proper faces,
    then this isomorphism holds in positive degrees.
  \end{enumroman}
\end{proposition}

Observe that any minimal face~\(P\in\PPP\) is a closed manifold
and moreover orientable as \(X\) and therefore also \(\piinv{P}\) are so,
\cf~\cite[Lemma~5.4]{AyzenbergEtAl:2014}.
Thus, the condition~\(H^{0}(\B^{*}(P))=\kk\) in part~\ref{pr1} above
implies that such a ~\(P\) is a single point.
Hence \(X\) is of dimension~\(2r\), any face contains a vertex
and the \(1\)-skeleton of any face is connected. These conditions are also sufficient.

\begin{proof}
  \def\AAA{H^{0}(\AB^{*}(X))}
  \ref{pr1}
  By assumption and Theorem\nobreakspace \ref {thm:iso-HABc-HBc}, we have
  \begin{equation}
    H^{0}(\AB^{*}(X)) = \bigoplus_{P\in\PPP} R_{P}\,t_{P}.
  \end{equation}
  Inspection of the formula
  given in Proposition\nobreakspace \ref {thm:product-AB0} implies that \(\AAA\) is multiplicatively closed.

  The same formula shows that
  the morphism~\(\phi\colon\kk[\,t_{P}:P\in\PPP\,]\to\AAA\)
  sending the indeterminate~\(t_{P}\) to~\(t_{P}\in R_{P}t_{P}\)
  factors through~\(\kk[\PPP]\).
  Comparing \eqref{eq:face-ring-R-module} with~\eqref{eq:iso-RP-RtF},
  we see that the resulting map~\(\bar\phi\colon\kk[\PPP]\to\AAA\)
  is a morphism of \(R\)-algebras.

  Moreover, \(\bar\phi\) is injective
  as its composition with the inclusion~\(\AAA\hookrightarrow\AB^{0}(X)=\HT^{*}(X^{T})\)
  is so by~\cite[Lemma~5.6]{MasudaPanov:2006}.
  Let \(F\facet P\) and let \(f\in R_{P}\).
  The product rule in~\(\AAA\)
  gives \((f t_{P})\cdot t_{F} = (f t_{F})\,t_{P}\). 
  This implies that \(\bar\phi\) is also surjective.

  \ref{pr2}
  This is essentially identical. The only difference is that now
  \(H^{0}(\B^{*}(Q))\) may be larger than \(\kk\) for~\(Q=X/T\). Still,
  \(H^{0}(\B^{*}(Q))\otimes R_{Q}t_{Q}\) is concentrated in degree~\(0\)
  since \(R_{Q}t_{Q}=\kk\), so that higher degrees are not affected.
\end{proof}

\section{Relation to previous work}
\label{sec:previous-work}

In this section we point out several connections
between our results and previous work
and show that many results concerning
the freeness and torsion-freeness of equivariant cohomology
have natural extensions to all syzygy orders.
The blanket assumptions on \(T\)-manifolds stated in Section\nobreakspace \ref {sec:blowup}
remain in force.

\subsection{Cohomology of orbit spaces}

Let \(X\) be an orientable rational cohomology manifold~\(X\) with a \(T\)-action.
In~\cite{Bredon:1974}, Bredon studied the cohomology
of the quotient~\(\Xf/T\) of the locally free part~\(\Xf=X_{r}\setminus X_{r-1}\)
under the assumption that \(\HTc^{*}(X)\) is free over~\(R\) (that is, of depth~\(r\)).
One of Bredon's results
is that \(H^{i}(\Xf/T)\) vanishes for~\(i>\dim X-2r\) \cite[Cor.~2]{Bredon:1974}.
(Note that \(H^{*}(\Xf/T)=H^{*}(X/T)\) if \(X\) is a regular \(T\)-manifold, see~\eqref{eq:P-incl-int}.)
Observing that compared to~\cite{Bredon:1974}
we have reversed the roles of compact and closed supports,
we can formulate a version for regular \(T\)-manifolds and other depths of~\(\HT^{*}(X)\) as follows.

\begin{corollary}
  Let \(X\) be an orientable regular \(T\)-manifold, and let \(j\ge0\).
  If \(\depth\HT^{*}(X)\ge j\), then \(\Hc^{i}(X/T)=0\) for~\(i>\dim X-r-j\).
\end{corollary}

\begin{proof}
  On the \(E_{2}\)~page
  of the spectral sequence~\eqref{eq:ss-P-E1} converging to~\(\Hc^{*}(X/T)\),
  the \(i\)-th column is \(H^{i}(\Bc^{*}(X/T))\).
  The degree shift in the definition of~\(\Bc^{i}(X/T)\) implies
  that this graded vector space is concentrated in degrees~\(\le\dim X-2r\).
  
  If \(\depth\HT^{*}(X)\ge j\), then \(H^{i}(\ABc^{*}(X))\) vanishes for~\(i>r-j\)
  by Proposition\nobreakspace \ref {thm:depth-ABc}\,\ref{thm:depth-ABc-hommf},
  hence also \(H^{i}(\Bc^{*}(X/T))\) by Theorem\nobreakspace \ref {thm:iso-HABc-HBc}.
  Thus, \(\Hc^{*}(X/T)\) vanishes in degrees greater than~\(\dim X-r-j\).
\end{proof}

\subsection{Torus manifolds}

If \(X\) satisfies Poincaré duality,
then this imposes strong restrictions
on the syzygies that can appear in equivariant cohomology,
see~\citeorbitsone{Prop.~5.12}:

\begin{theorem}[Allday--Franz--Puppe]
  Let \(X\) be a compact orientable \(T\)-manifold.
  If \(\HT^{*}(X)\) is a syzygy of order~\(\ge r/2\),
  then it is free over~\(R\).
\end{theorem}

In this section we prove a stronger result
for compact orientable \(T\)-manifolds of dimension~\(2r\),
that is, twice the dimension of~\(T\).
We start by observing that dimensions below~\(2r\)
are not interesting from the point of view of syzygies.

\begin{lemma}
  \label{thm:torsion-free-regular}
  Let \(X\) be a \(T\)-manifold of dimension~\(\le 2r\).
  If \(\HT^{*}(X)\) is torsion-free,
  then \(\dim X=2r\), \(X^{T}\ne\emptyset\), and the action is locally standard.
\end{lemma}

\begin{proof}
  If \(\HT^{*}(X)\) is torsion-free, then any component~\(Y\) of the fixed-point set~\(X^{K}\)
  of any subtorus~\(K\subset T\) contains a \(T\)-fixed point,
  \cf~\cite[Prop.~4.13]{AlldayFranzPuppe:orbits4}.
  With our tools this can be seen as follows: Write \(L=T/K\).
  By Proposition\nobreakspace \ref {thm:syzygy-depth-fixedpoints}\,\ref{thm:syzygy-XK},
  \(H_{L}^{*}(Y)\subset H_{L}^{*}(X^{K})\) is torsion-free over~\(H^{*}(BL)\),
  hence injects into~\(H_{L}^{*}(Y^{T})\) by Theorem\nobreakspace \ref {thm:3:intro-partial-exactness}.

  Because the codimension of~\(X^{T}\) in~\(X\) is at least~\(2r\),
  there cannot be any fixed points if \(\dim X<2r\).
  The same conclusion holds if \(\dim X^{K}<2r-2\) for some characteristic circle~\(K\).
  So the \(T\)-action is regular with fixed points and therefore locally standard
  by Lemma\nobreakspace \ref {thm:regular-loc-std}.
\end{proof}

\begin{remark}
  If \(X\) is a \(T\)-manifold of dimension~\(2r\),
  then \(X^{T}\) 
  is discrete and corresponds to the vertices (\(=\)~rank~\(0\) faces) of~\(X/T\).  
  Hence \(\HT^{*}(X)\) is free over~\(R\) if and only if \(H^{*}(X)\)
  vanishes in odd degrees:
  If \(\Hodd(X)=0\), then the Serre spectral sequence converging to~\(\HT^{*}(X)\)
  collapses at the \(E_{2}\)~page,
  hence \(\HT^{*}(X)\cong H^{*}(X)\otimes R\) is free over~\(R\)
  by the Leray--Hirsch theorem.
  Conversely, if \(\HT^{*}(X)\) is free over~\(R\),
  then \(\dim\Hodd(X)=\dim\Hodd(X^{T})=0\),
  see~
  \cite[Sec.~3.1]{AlldayPuppe:1993}.
\end{remark}

A \(2r\)-dimensional compact orientable \(T\)-manifold~\(X\)
is called a \emph{torus manifold} if the action (is effective and) has fixed points,
\cf~\cite[Sec.~3]{MasudaPanov:2006}.
Smooth complete toric varieties are examples of locally standard torus manifolds,
\cf~Section\nobreakspace \ref {sec:3:toric}.
For a locally standard torus manifold~\(X\), Masuda--Panov proved that
\(\HT^{*}(X;\Z)\) is free over~\(H^{*}(BT;\Z)\) if and only if every face of~\(X/T\) is acyclic over~\(\Z\)
\cite[Thm.~9.3]{MasudaPanov:2006}.
This generalized the description of~\(\HT^{*}(X;\Z)\)
in terms of Stanley--Reisner rings given
by Davis--Januszkiewicz~\cite[Thm.~4.8]{DavisJanuszkiewicz:1991}
for quasi-toric manifolds.
With our methods,
we can recover this result (for cohomology with real coefficients)
and also relate it to torsion-freeness.
We will consider non-compact and non-orientable \(T\)-manifolds
as well as other syzygies in the next section.

For the rest of this section, \(X\) denotes a torus manifold (of dimension~\(2r\)).

\begin{lemma}
  \label{thm:HQ-HTX-inj}
  Assume the \(T\)-action on~\(X\) to be locally standard with quotient \(Q=X/T\).
  Then the canonical map~\(\rho\colon H^{*}(Q)\to\HT^{*}(X)\) is injective,
  and \(\rho(\tilde H^{*}(Q))\) is contained in the \(R\)-torsion submodule of~\(\HT^{*}(X)\).  
\end{lemma}

\begin{proof}
  See the discussion preceding eq.~(4.2) and the proof of Prop.~5.3 in~\cite{AyzenbergEtAl:2014}.
  (That \(X\) has fixed points, as required by~\cite{AyzenbergEtAl:2014},
  is not needed for these arguments.)
  It also follows from the case~\(m=2\) of Lemma\nobreakspace \ref {thm:inj-barE-E} below
  because \(\tilde H^{*}(Q)\) is contained in~\(H^{-1}(\barB^{*}(Q))\)
  as \(X^{T}\) is discrete,
  and the torsion submodule of~\(\HT^{*}(X)\) equals \(H^{-1}(\barAB^{*}(X))\)
  by the localization theorem in equivariant cohomology.
\end{proof}

\begin{lemma}
  \label{thm:HB-face-acyclic}
  If \(X\) is locally standard and \(X/T\) face-acyclic,
  then for any~\(P\in\PPP\),
  \begin{equation*}
    H^{i}(\B^{*}(P)) = H^{i}(P) =
    \begin{cases}
      \kk & \text{if \(i=0\),} \\
      0 & \text{if \(i>0\).}
    \end{cases}
  \end{equation*}
\end{lemma}

Here we say that a manifold with corners~\(Q\) is \emph{face-acyclic}
if all of its faces (including \(Q\) itself) are acyclic.
The complex~\(\B^{*}(P)\) was defined in~\eqref{eq:ss-P-E1}.

\begin{proof}
  We can look at \(\B^{*}(P)\)
  as the \(E_{1}\)~page of the spectral sequence induced by the filtration of~\(P\) by its faces
  and converging to~\(H^{*}(P)\).
  Since \(X\) is orientable, so is \(\piinv{Q}\) and therefore also \(\dot Q\) for any~\(Q\in\PPP\),
  \cf~\cite[Lemma~5.4]{AyzenbergEtAl:2014}.
  If \(Q\) is of rank~\(i\), then it is also of dimension~\(i\),
  hence by Poincaré duality \(\Hc^{*}(\dot Q)\cong H_{i-*}(\dot Q)=H_{i-*}(Q)\)
  is concentrated in degree~\(i\).
  Thus, each~\(\B^{i}(P)\) is concentrated in degree~\(0\), and
  the spectral sequence degenerates at the \(E_{2}\)~level.
  This proves the claim.
\end{proof}

\begin{proposition}
  \label{thm:torus-manifold-free}
  The following conditions are equivalent for a torus manifold~\(X\):
  \begin{enumarabic}
  \item \label{ft1}
    \(\HT^{*}(X)\) is free over~\(R\).
  \item \label{ft2}
    \(\HT^{*}(X)\) is torsion-free over~\(R\).
  \item \label{ft3}
    The \(T\)-action on~\(X\) is locally standard, and
    \(Q=X/T\) is face-acyclic.
  \end{enumarabic}
\end{proposition}

The equivalence of~\ref{ft1} and~\ref{ft2} is Corollary\nobreakspace \ref {thm:torus-manifold-free-intro}.

\begin{proof}
  \(\ref{ft1}\Rightarrow\ref{ft2}\) is trivial.

  \(\ref{ft2}\Rightarrow\ref{ft3}\):
  The action is locally standard by Lemma\nobreakspace \ref {thm:torsion-free-regular}.
  Let \(P\) be a face of~\(X/T\). 
  By Lemma\nobreakspace \ref {thm:HQ-HTX-inj}, 
  \(\tilde H^{*}(P)\) injects into the torsion submodule of~\(H_{T^{P}}^{*}(X^{P})\).
  Since this submodule is trivial by Proposition\nobreakspace \ref {thm:syzygy-depth-fixedpoints}\,\ref{thm:syzygy-XK}, \(P\) must be acyclic.
  
  \(\ref{ft3}\Rightarrow\ref{ft1}\):
  Because \(X\) is compact orientable and all faces acyclic,
  \(H^{i}(\B^{*}(P))\) vanishes for any~\(P\in\PPP\) and any~\(i>0\) by Lemma\nobreakspace \ref {thm:HB-face-acyclic}.
  Hence the cohomology of the Atiyah--Bredon sequence~\(\AB^{*}(X)\)
  is concentrated in the column~\(i=0\) by Theorem\nobreakspace \ref {thm:iso-HABc-HBc},
  and the augmented Atiyah--Bredon sequence~\(\barAB^{*}(X)\)
  is exact except possibly at the positions~\(i=-1\) and \(i=0\).
  But this implies that it is exact everywhere \citeorbitsone{Lemma~5.6},
  which by Theorem\nobreakspace \ref {thm:3:intro-partial-exactness} shows that \(\HT^{*}(X)\) is free.
\end{proof}

In dimension greater than~\(2r\),
torsion-freeness of~\(\HT^{*}(X)\) does not imply freeness,
see Section\nobreakspace \ref {sec:3:mutants} for an example and further references.

\begin{corollary}
  \label{thm:iso-HTX-face-ring}
  Assume that the \(T\)-action on~\(X\) is locally standard and \(X/T\) face-acyclic.
  Then \(\HT^{*}(X)\cong\kk[\PPP]\) as \(R\)-algebras.
\end{corollary}

\begin{proof}
  By Proposition\nobreakspace \ref {thm:torus-manifold-free}, \(\HT^{*}(X)\) is free over~\(R\), hence isomorphic
  to~\(H^{0}(\AB^{*}(X))\) by Theorem\nobreakspace \ref {thm:3:intro-partial-exactness}.
  The claim thus follows from Proposition\nobreakspace \ref {thm:H0AB-face-ring}\,\ref{pr1}.
\end{proof}

We conclude this section by a description of~\(\HT^{*}(X)\)
in the case where all proper faces of~\(X/T\) are acyclic.
For this we need some preparations.
In addition to the requirements stated so far,
we assume that the \(T\)-action on~\(X\) is locally standard,
and we write \(Q=X/T\).

As discussed in the proof of Lemma\nobreakspace \ref {thm:HB-face-acyclic},
the complex~\(\B^{*}(Q)\) 
is the \(E_{1}\)~page of a spectral sequence and converging to~\(H^{*}(Q)\).
Analogously,
the non-augmented Atiyah--Bredon complex~\eqref{eq:3:atiyah-bredon} 
is the \(E_{1}\)~page of the spectral sequence arising from the orbit filtration
and converging to~\(\HT^{*}(X)\). Moreover, the augmented complex~\(\barAB^{*}(X)\)
is the \(E_{1}\)~page of the spectral sequence arising from the induced filtration
\begin{equation}
  \label{eq:filtration-CX}
  X \cup \{*\} \subset X \cup C X_{0} \subset \dots \subset X \cup C X_{r} = CX
\end{equation}
of the cone~\(CX\) of~\(X\) with apex~\(*\)
and converging to~\(\HT^{*}(CX,*)=0\), see \citeorbitsone{Lemma~5.6}
and its proof. Let us write~\(E_{m}\) for the \(m\)-th page of this spectral sequence.

In the same fashion,
we augment \(\B^{*}(Q)\) to a complex~\(\barB^{*}(Q)\)
by setting~\(\barB^{-1}(Q)=H^{*}(Q)\), the differential~\(\barB^{-1}(Q)\to\barB^{0}(Q)\)
being the restriction to the rank-\(0\) faces. Again, this is the page~\(\bar E_{1}\)
of a spectral sequence~\((\bar E_{m})\) converging to~\(0\).
Because the map~\((CX,*)\to(CQ,*)\) is compatible with the filtrations,
we get an induced map of spectral sequences~\(\kappa_{m}\colon\bar E_{m}\to E_{m}\).

\begin{lemma} 
  \label{thm:inj-barE-E}
  The map \(\kappa_{m}\colon\bar E_{m}\to E_{m}\) is injective for any~\(m\ge1\).
\end{lemma}

\begin{proof}
  Because the boundary~\(\partial Q\) is collared in~\(Q\) \cite[Thm.~2]{Brown:1962},
  the inclusion~\(\dot Q\hookrightarrow Q\) is a homotopy equivalence.
  Hence the principal \(T\)-bundle~\(\Xf\to\dot Q\) is the restriction of some topological
  principal \(T\)-bundle~\(\tX\to Q\). Moreover, the inclusion~\(\Xf\to X\) extends to a continuous
  \(T\)-equivariant surjection~\(\psi\colon\tX\to X\).\footnote{%
    I have learnt this construction from Anton Ayzenberg.}
  Note that the projection~\(\tX\to Q\) induces an isomorphism \(H^{*}(Q)=\HT^{*}(\tX)\).

  Lift the orbit filtration of~\(X\) to a filtration of~\(\tX\) via~\(\psi\) and extend it to the cone~\(C\tX\)
  as in~\eqref{eq:filtration-CX}.
  Let \((\tilde E_{m})\) be the induced spectral sequence converging to \(\HT^{*}(C\tX,*)=0\),
  so that we get a map of spectral sequences
  \begin{equation}
    \bar E_{m} \stackrel{\kappa_{m}}\longrightarrow E_{m} \longrightarrow \tilde E_{m}.
  \end{equation}
  Because \(T\) acts freely on~\(C\tX\setminus\{*\}\), the composed map is bijective for \(m=1\),
  \cf~\cite[Prop.~2.7]{AlldayFranzPuppe:orbits4},
  hence for all~\(m\ge1\). This implies that \(\kappa_{m}\) is injective.
\end{proof}

\begin{proposition}
  \label{thm:product-HTX}
  Let \(X\) be a locally standard torus manifold.
  If all proper faces of~\(Q=X/T\) are acyclic, then there is an isomorphism of \(R\)-algebras
  \begin{equation*}
    \HT^{*}(X)\cong\kk[\PPP]\oplus\tilde H^{*}(Q).
  \end{equation*}
\end{proposition}

Here both \(\kk[\PPP]\) and~\(\tilde H^{*}(Q)\) are equipped with a product, and any mixed product
vanishes unless one factor lies in~\(\kk\subset\kk[\PPP]\).
The \(R\)-module structure on~\(\kk[\PPP]\) has been defined in Section\nobreakspace \ref {sec:product}.
The one on~\(\tilde H^{*}(Q)\) is induced by~\(\HT^{*}(X)\) via Lemma\nobreakspace \ref {thm:HQ-HTX-inj};
a different description will be given in the proof below.

Ayzenberg--Masuda--Park--Zeng~\cite[Prop.~5.2 \& 5.3]{AyzenbergEtAl:2014}
proved an integral version of Proposition\nobreakspace \ref {thm:product-HTX}
for locally standard torus manifolds
under the additional assumption that the principal \(T\)-bundle
over the interior of~\(Q\) is trivial.

\begin{proof}
  We consider the sequence
  \begin{equation}
    \HT^{*}(X,\Xf) \longrightarrow \HT^{*}(X) \longrightarrow \HT^{*}(\Xf)
  \end{equation}
  where \(\Xf=X\setminus X_{r}\) is the free part of the action.
  We have \(\HT^{*}(\Xf)=H^{*}(\dot Q)=H^{*}(Q)\).
  As in~\cite[eq.~(5.2)]{AyzenbergEtAl:2014}, this implies
  that the above sequence splits via the map~\(\rho\colon H^{*}(Q)\to\HT^{*}(X)\), so that
  \begin{equation}
    \HT^{*}(X) = \HT^{*}(X,\Xf) \oplus \rho(H^{*}(Q))
    = \kk \oplus \HT^{*}(X,\Xf) \oplus \rho(\tilde H^{*}(Q))
  \end{equation}
  as graded vector spaces.
  By Lemma\nobreakspace \ref {thm:HQ-HTX-inj},
  \(\rho(\tilde H^{*}(Q))\) is contained in the \(R\)-torsion submodule~\(M\) of~\(\HT^{*}(X)\).

  It is enough to show that \(\rho(\tilde H^{*}(Q))\) equals \(M\)
  and that the map~\(\HT^{*}(X)/M\to H^{0}(\AB^{*}(X))\) is bijective in positive degrees:
  Then \(\HT^{*}(X,\Xf)\) has to be isomorphic to the positive degree part of~\(H^{0}(\AB^{*}(X))\),
  which is the positive degree part of~\(\kk[\PPP]\) by Proposition\nobreakspace \ref {thm:H0AB-face-ring}\,\ref{pr2}.
  Moreover, because \(\HT^{*}(X,\Xf)\) and \(M\) are both ideals of~\(\HT^{*}(X)\), all non-trivial mixed products
  have to vanish, while the products with both factors from the same summand or with one factor from~\(\kk\)
  are already known.

  Consider the cohomology~\(H^{*}(\barAB^{*}(X))=E_{2}\) of the augmented Atiyah--Bredon complex.
  Because all proper faces of~\(Q\) are acyclic, the \(i\)-th column is isomorphic to~\(H^{i}(\barB^{*}(Q))\)
  for~\(i\ge 1\) by Theorem\nobreakspace \ref {thm:iso-HABc-HBc}.
  Recall that the spectral sequence~
  \((\bar E_{m})\) converges to~\(0\).
  Since the map~\(\bar E_{m}\to E_{m}\) is injective for~\(m\ge1\) by Lemma\nobreakspace \ref {thm:inj-barE-E},
  it must be bijective for~\(m\ge2\) because there
  is no possible target in~\(E_{m}\) for a higher differential that is not the image of a differential in~\(\bar E_{m}\).
  This proves the claims made above.
\end{proof}

\subsection{\texorpdfstring{\boldmath{Locally standard \(T\)-manifolds of dimension~\(2r\)}}{Locally standard T-manifolds of dimension 2r}}

In this section, \(X\) denotes a
(not necessarily compact or orientable) locally standard \(T\)-manifold of dimension~\(2r\).
In the orientable case,
this includes smooth toric varieties and, more generally,
locally standard open torus manifolds in the sense of Masuda~\cite[Sec.~2]{Masuda:2006}.
If one imposes the additional condition
that \(X\) be compact orientable with non-empty fixed point sets,
then one gets exactly the locally standard torus manifolds
mentioned in the previous section.
Toric origami manifolds with co-orientable fold are another class of examples,
which contains also non-orientable manifolds \cite[Rem.~2.5 \& Lemma~5.1]{HolmPires:2013}.

\begin{lemma}
  \label{thm:BB-boundary}
  Let \(P\in\PPP\) be face-acyclic and of rank~\(p\).
  Then for all~\(i\),
  \begin{equation*}
    H^{i}(\Bc^{*}(P;\kktilde)) = \Hc^{i}(P;\kktilde) \cong \tilde H_{p-i-1}(\partial P).
  \end{equation*}
\end{lemma}

In particular, each \(H^{i}(\Bc^{*}(P;\kktilde))\) is concentrated in degree~\(0\).

\begin{proof}
  The isomorphism~\(H^{i}(\Bc^{*}(P;\kktilde)) = \Hc^{i}(P;\kktilde)\)
  is shown as in Lemma\nobreakspace \ref {thm:HB-face-acyclic}. Moreover,
  the long exact homology sequence for the pair~\((P,\partial P)\)
  and Poincaré--Lefschetz duality~\cite[Thm.~11.3]{Massey:1978} imply
  \begin{equation}
    \label{eq:iso-partialP-P}
    \tilde H_{p-i-1}(\partial P) = H_{p-i}(P,\partial P)
    \cong \Hc^{i}(P;\kktilde).
    \qedhere
  \end{equation}
\end{proof}

\begin{proposition}
  \label{thm:open-torus-manifold-syzygy}
  Let \(j\ge1\).
  Then \(\HT^{*}(X)\) is a \(j\)-th syzygy if and only if each face~\(P\in\PPP\) 
  is acyclic and additionally \(\tilde H_{i}(\partial P)=0\) for~\(i<\min(\rank P,j)-1\)
  in case \(P\) is not compact or \(X\) not orientable.
\end{proposition}

Unlike in the compact case, syzygies of any order~\(\le r\) can appear as
the equivariant cohomology of non-compact \(T\)-manifolds of dimension~\(2r\),
see Section\nobreakspace \ref {sec:examples-non-compact}.

\begin{proof}
  By Theorem\nobreakspace \ref {thm:condition-syzygy-HBc}, \(\HT^{*}(X)\) is a \(j\)-th syzygy if and only if
  \begin{equation}
    \label{eq:quotient-criterion-free}
    H^{i}(\Bc^{*}(P;\kktilde))=0 \;\;\;\hbox{for any face~\(P\) of rank~\(p>0\) and any~\(i>\max(p-j,0)\)}.
  \end{equation}
  
  We start with the case~\(j=1\). 
  We have to show that condition~\eqref{eq:quotient-criterion-free}
  is equivalent to all faces of positive rank being acyclic and with non-empty boundary.

  \(\Rightarrow\):
  Assume that \(P\) is a face of minimal rank~\(p>0\) such that 
  \(\partial P=\emptyset\). Then \(H^{p}(\Bc^{*}(P;\kktilde))=\Hc^{*+p}(\dot P;\kktilde)\ne0\),
  which is a contradiction.
  Now assume that \(P\) is a non-acyclic face of minimal rank~\(p>0\).
  Then \(\Hc^{*}(\dot P;\kktilde)\cong H_{p-*}(P)\) is not concentrated in degree~\(p\).
  Hence \(\Bc^{p}(P;\kktilde)\) is the only piece of~\(\Bc^{*}(P;\kktilde)\)
  that does not vanish in positive degrees. Again, \(H^{p}(\Bc^{*}(P;\kktilde))\ne0\).

  \(\Leftarrow\): 
  Let \(P\) be a face of rank~\(p>0\). Since \(\partial P\ne\emptyset\), \(P\) must
  have a facet. This implies that the differential~\(\Bc^{p-1}(P;\kktilde)\to\Bc^{p}(P;\kktilde)=\kk\)
  is surjective (see Lemma\nobreakspace \ref {thm:poincare-boundary}), and \(H^{p}(\Bc^{*}(P;\kktilde))=0\).

  We now consider the case~\(j>1\). By what we have just done,
  we can assume for either direction of the proof that all faces are acyclic.
  By Lemma\nobreakspace \ref {thm:BB-boundary}, the vanishing condition in the claim is equivalent
  to~\eqref{eq:quotient-criterion-free}.
  
  \(\Rightarrow\): is therefore done.

  \(\Leftarrow\):
  It remains to show that any compact face~\(P\) may be omitted from~\eqref{eq:quotient-criterion-free}
  if \(X\) is orientable.
  In this case \(\piinv{P}\) is compact orientable as well, and
  since \(P\) is face-acyclic, \(H_{T^{P}}^{*}(\piinv{P})\)
  is free over~\(R_{T^{P}}\) by Proposition\nobreakspace \ref {thm:torus-manifold-free}. 
  By the already established implication~``\(\Rightarrow\)'',
  this proves \eqref{eq:quotient-criterion-free} for this~\(P\) (and its faces).
\end{proof}

\begin{corollary}
  \label{thm:open-torus-manifold}
  $ $
  \begin{enumroman}
  \item \label{bb1}
    \(\HT^{*}(X)\) is torsion-free if and only if
    every face~\(P\in\PPP\) is acyclic and contains a vertex.
  \item \label{bb2} 
    Assume \(X\) to be orientable. Then
    \(\HT^{*}(X)\) is free over~\(R\)
    if and only if every face~\(P\in\PPP\) is acyclic and \(\partial P\) is acyclic
    for every non-compact face~\(P\in\PPP\).
  \end{enumroman}
\end{corollary}

For open torus manifolds, part~\ref{bb2} is due to Masuda~\cite[Thm.~4.3]{Masuda:2006};
see also Goertsches--Töben~\cite[Sec.~10]{GoertschesToeben:2010}.
In the context of orientable rational cohomology manifolds,
Bredon~\cite[Cor.~3]{Bredon:1974} established much earlier that
the freeness of~\(\HT^{*}(X)\) implies the acyclicity of~\(\Xf/T\).

\begin{proof}
  \ref{bb1} A minimal face of positive rank without vertex has empty boundary.
  That any face contains a vertex is therefore equivalent to any face of positive rank
  having non-empty boundary. Hence we recover the case~\(j=1\) of Proposition\nobreakspace \ref {thm:open-torus-manifold-syzygy}.

  \ref{bb2} This is the case~\(j=r\).
\end{proof}

The ``only if'' parts in Proposition\nobreakspace \ref {thm:open-torus-manifold-syzygy}
and Corollary\nobreakspace \ref {thm:open-torus-manifold}
do need the dimension condition on~\(X\).
(Just take the product of a toric manifold and some compact manifold with trivial \(T\)-action.)

\subsection{Toric varieties}
\label{sec:3:toric}

Let \(\Sigma\) be a regular rational fan in~\(\R^{r}\).
The associated smooth toric variety~\(X_{\Sigma}\)
is an orientable locally standard \(T\)-manifold of dimension~\(2r\),
so the results of the preceding sections apply.

There is an inclusion-reversing bijection between
the faces of~\(X_{\Sigma}/T\) and the cones~\(\sigma\in\Sigma\);
the face~\(P_{\sigma}\in\PPP\) corresponding to a cone~\(\sigma\in\Sigma\)
is of dimension \(\dim P_{\sigma}=\codim\sigma\) and acyclic as
\(\dot P_{\sigma}\) is diffeomorphic to~\(\R^{\codim\sigma}\).

In particular, if \(X_{\Sigma}\) is compact,
then by Corollary\nobreakspace \ref {thm:iso-HTX-face-ring} there is
an isomorphism of \(R\)-algebras
\begin{equation}
  \label{eq:HTXSigma-SR}
  \HT^{*}(X_{\Sigma}) \cong \kk[\Sigma],
\end{equation}
where the Stanley--Reisner ring~\(\kk[\Sigma]\) is isomorphic to the face ring of~\(\PPP\).
In fact, \eqref{eq:HTXSigma-SR} holds for any smooth toric variety,
as shown first by Bifet--De\,Concini--Procesi~\cite[Thm.~8]{BifetDeConciniProcesi:1990}.
In this section we want to reformulate the criterion given in 
Proposition\nobreakspace \ref {thm:open-torus-manifold-syzygy}
in terms of the links of the cones in~\(\Sigma\),
and we determine the depth and \(\Ext\)~modules
of Stanley--Reisner rings.

The fan~\(\Sigma\) canonically determines a simplicial complex
whose \(k\)-simplices are the cones~\(\sigma\in\Sigma\) of dimension~\(k+1\).
When we speak of the link~\(L_{\sigma}\) of~\(\sigma\), we mean
its link in this simplicial complex (which equals the whole complex for~\(\sigma=0\)).
Note that \(\dim L_{\sigma}=\codim\sigma-1\).

\begin{lemma}
  \label{thm:link-boundary}
  \( 
    H^{i}(\BB^{*}(P_{\sigma})) \cong \tilde H_{\codim\sigma-1-i}(L_{\sigma})
  \) 
  for any \(\sigma\in\Sigma\) and any~\(i\).
\end{lemma}

\begin{proof}
  Let \(d=\dim\sigma\), and
  let \(\rho_{1}\),~\dots,~\(\rho_{s}\) be the rays of~\(\Sigma\) not contained in~\(\sigma\).
  The \(m\)-simplices in~\(L_{\sigma}\) (including the empty simplex of dimension~\(-1\))
  correspond to the \((m+d+1)\)-dimensional cones in~\(\Sigma\) containing~\(\sigma\),
  as do the \((m+1)\)-codimensional faces of~\(P_{\sigma}\).
  Let us write the simplex corresponding to such a cone~\(\tau\)
  as~\((i_{1},\dots,i_{m})\) where \(\rho_{i_{1}}\),~\dots,~\(\rho_{i_{m}}\)
  are the rays of~\(\tau\) not contained in~\(\sigma\).

  In the augmented simplicial chain complex~\(\tilde C_{*}(L_{\sigma})\), we have
  \begin{equation}
    d (i_{1},\dots,i_{m}) = \sum_{k=1}^{m} (-1)^{k-1} (i_{1},\dots,i_{k-1},i_{k+1},\dots,i_{m}).
  \end{equation}
  On the other hand,
  since all faces are acyclic, there is a canonical basis~\((c_{\tau})\) of~\(\BB^{*}(P)\)
  corresponding to the~\(P_{\tau}\)'s, and the definition~\eqref{eq:definition-d-BB} of the differential gives
  \begin{equation}
    d c_{(i_{1},\dots,i_{m})} = \sum_{k=1}^{m} (-1)^{k-1} c_{(i_{1},\dots,i_{k-1},i_{k+1},\dots,i_{m})}.
  \end{equation}
  Therefore, the complexes~\(\tilde C_{*}(L_{\sigma})\) and~\(\BB^{\codim\sigma-1-*}(P_{\sigma})\) are isomorphic,
  whence the claim.
\end{proof}

Combining Theorem\nobreakspace \ref {thm:iso-HABc-HBc} with Lemmas~\ref{thm:Bc-BB-iso} and~\ref{thm:link-boundary},
we get for any~\(i\ge0\) an isomorphism of graded vector spaces 
\begin{equation}
  \label{eq:decomposition-toric}
  H^{i}(\ABc^{*}(X_{\Sigma}))
  = \bigoplus_{\sigma\in\Sigma} \tilde H_{\codim\sigma-i-1}(L_{\sigma})\otimes R_{P_{\sigma}} t_{P_{\sigma}}.
\end{equation}
A similar decomposition has been obtained by
Barthel--Brasselet--Fieseler--Kaup
for the equivariant intersection cohomology
of an arbitrary toric variety
\cite[eq.~(3.4.2) \& Rem.~3.5]{BarthelBrasseletFieselerKaup:2002}.

\begin{proposition} $ $
  \label{thm:SR}
  \begin{enumroman}
  \item \label{ext-SR}
    For any~\(i\ge0\) there is an isomorphism of graded vector spaces
    \begin{equation*}
      \Ext_{R}^{i}(\kk[\Sigma],R) \cong
      \bigoplus_{\sigma\in\Sigma} \tilde H_{\codim\sigma-i-1}(L_{\sigma})\otimes R_{P_{\sigma}} t_{P_{\sigma}}[-2r].
    \end{equation*}
  \item \label{depth-SR}
    One has \(\depth_{R} \kk[\Sigma]\ge i\) if and only if \(\tilde H_{j}(L_{\sigma}) = 0\)
    for all~\(\sigma\in\Sigma\) and all~\(j<\codim\sigma-i-1\).
  \end{enumroman}
\end{proposition}

In the case where \(X_{\Sigma}\) is a toric subvariety of~\(\C^{r}\),
so that \(\kk[\Sigma]\) is a quotient of~\(R\),
a decomposition as in part~\ref{ext-SR} has been obtained by
Yanagawa~\cite[Prop.~3.1]{Yanagawa:2000}
and in an equivalent form using local cohomology by Mustaţă~\cite[Cor.~2.2]{Mustata:2000}.
In the same setting, part~\ref{depth-SR} is due to Munkres~\cite[Thm.~3.1]{Munkres:1984};
the general case could be deduced from this.

\begin{proof}
  Recall that \(\HT^{*}(X_{\Sigma})\cong\hHTc_{*}(X_{\Sigma})[2r]\)
  by equivariant Poincaré duality~\eqref{eq:PD-c},
  where ``\([2r]\)'' denotes a degree shift by~\(2r\).
  Combining this with~\eqref{eq:HTXSigma-SR} and~\eqref{eq:AB-Ext-R-c},
  we obtain
  \begin{equation}
    \Ext_{R}^{i}(\kk[\Sigma],R) = H^{i}(\ABc^{*}(X_{\Sigma}))[-2r],
  \end{equation}
  which together with~\eqref{eq:decomposition-toric} proves the first claim.
  The second one now follows from~\eqref{eq:depth-Ext}.
\end{proof}

\begin{proposition}
  \label{thm:syzygy-SR}
  Let \(\Sigma\) be a regular rational fan in~\(\R^{r}\)
  with associated smooth toric variety~\(X_{\Sigma}\), and let \(j\ge0\).
  Then \(\HT^{*}(X_{\Sigma})\) is a \(j\)-th syzygy if and only if
  \(\tilde H_{i}(L_{\sigma})=0\)
  for all~\(\sigma\in\Sigma\) and all~\(i<\min(j,\codim\sigma)-1\).
\end{proposition}

\begin{proof}
  Use Theorem\nobreakspace \ref {thm:condition-syzygy-HBB}, Lemma\nobreakspace \ref {thm:link-boundary} and the identity~\(\codim\sigma=\rank P_{\sigma}\).
\end{proof}

For the freeness of~\(\HT^{*}(X_{\Sigma})\) both Proposition\nobreakspace \ref {thm:SR}\,\ref{depth-SR} and Proposition\nobreakspace \ref {thm:syzygy-SR} reduce
to Reisner's criterion for the Cohen--Macaulayness of Stanley--Reisner rings, \cf~\cite[Cor.~5.3.9]{BrunsHerzog:1998}
or~\cite[Cor.~3.4]{Munkres:1984}.
(\(\kk[\Sigma]\) is free over~\(R\) if and only if it is a Cohen--Macaulay \(R\)-module of dimension~\(r\).
Because \(\kk[\Sigma]\cong\HT^{*}(X_{\Sigma})\) is finitely generated over~\(R\),
it is Cohen--Macaulay of dimension~\(r\) over~\(R\) if and only if it is so
over itself \cite[Lemma~2.6]{GoertschesRollenske:2011}.)
See also Barthel--Brasselet--Fieseler--Kaup~\cite[Lemma~4.5]{BarthelBrasseletFieselerKaup:2002}.

Also note that the depth and the syzygy order of~\(\R[\Sigma]\) only depend on the combinatorics of the fan~\(\Sigma\)
and not on its geometry. In fact, they only depend on the support~\(|\Sigma|\subset\R^{r}\)
because the reduced homology of a link~\(L_{\sigma}\) is, up to degree shift, the local homology of any interior point
of~\(\sigma\), \cf~\cite[Lemma~3.3]{Munkres:1984}.

\subsection{Equivariant injectivity}
\label{sec:equiv-inj}

Goertsches--Rollenske \cite[Sec.~5]{GoertschesRollenske:2011}
introduced the notion of equivariant injectivity for a compact \(T\)-manifold~\(X\).
If the fixed point set~\(X^{K}\) of every subgroup~\(K\subset T\) meets \(X^{T}\),
then \emph{equivariant injectivity} of the \(T\)-action
means that the restriction map~\(\HT^{*}(X)\to\HT^{*}(X^{T})\)
is injective or, equivalently by Theorem\nobreakspace \ref {thm:3:intro-partial-exactness},
that \(\HT^{*}(X)\) is torsion-free.
Goertsches--Rollenske~\cite[Thm.~5.9]{GoertschesRollenske:2011}
give two characterizations of equivariant injectivity.
In order to state them, we need to introduce another notion.

Assume \(r\ge1\), and let \(U\) be a \(T\)-stable tubular neighbourhood
of~\(X_{r-1}\setminus X_{r-2}\) in~\(X\setminus X_{r-2}\) with quotient~\(\bar U=U/T\).
Write \(P=X/T\), so that \(\dot P\) is the quotient of~\(X\setminus X_{r-1}\).
Then the \(T\)-action on~\(X\) \emph{has no essential basic cohomology} if the restriction map
\begin{equation}
  H^{*}(\dot P) \to H^{*}(\bar U\cap\dot P)
\end{equation}
is injective.
(In~\cite{GoertschesRollenske:2011} basic differential forms on the manifolds
are used to describe the cohomology of the orbit spaces, whence the name.)

Following \cite{GoertschesRollenske:2011}, we can characterize equivariant injectivity as follows:

\begin{proposition}
  The following are equivalent for a \(T\)-manifold~\(X\):
  \begin{enumarabic}
  \item \label{cc1}
    \(\HT^{*}(X)\) is torsion-free.
  \item \label{cc2}
    \(\depth_{R_{L}} H_{L}^{*}(X^{K})\ge1\)
    for every isotropy subtorus~\(K\subsetneq T\) occurring in~\(X\)
    with quotient~\(L=T/K\).
  \item \label{cc4}
    For every isotropy subtorus~\(K\subsetneq T\) occurring in~\(X\)
    with quotient~\(L=T/K\),
    the \(L\)-manifold~\(X^{K}\) has no essential basic cohomology.
  \end{enumarabic}
  If the \(T\)-action is regular, then these conditions are also equivalent to:
  \begin{enumarabic}[resume]
  \item \label{cc3}
    The restriction map
    \begin{equation*}
      H^{*}(P) \to \bigoplus_{Q\facet P} H^{*}(Q).
    \end{equation*}
    is injective for every face~\(P\) of~\(X/T\) of positive rank.
  \end{enumarabic}
\end{proposition}

Note that by Lemma\nobreakspace \ref {thm:reduction-effective}\,\ref{thm:reduction-effective-depth}
the depth condition in~\ref{cc2} can be reformulated as
\begin{equation}
  \depth_{R} \HT^{*}(\piinv{K}) \ge 1 + \dim K,
\end{equation}
which is the form in which it appears in~\cite{GoertschesRollenske:2011}.

\begin{proof}
  The equivalence \(\ref{cc1}\Leftrightarrow\ref{cc2}\)
  is the case~\(j=1\) of Proposition\nobreakspace \ref {thm:syzygy-depth-fixedpoints}\,\ref{thm:syzygy-depth-2}.

  Assume that \(T\) acts regularly.
  Condition~\ref{cc3}, applied to a face of rank~\(1\), shows that there
  are fixed points, so that the action is locally standard by Lemma\nobreakspace \ref {thm:regular-loc-std}.
  Moreover, the map in condition~\ref{cc3}
  is the transpose of the differential~\(\BB^{p-1}(P)\to\BB^{p}(P)\)
  defined in~\eqref{eq:definition-d-BB}, where \(p=\rank P\).
  Hence this condition is equivalent to~\ref{cc1}
  by Theorem\nobreakspace \ref {thm:condition-syzygy-HBB}.

  Let \(Q\facet P\), and let \(V\) be a \(T\)-invariant tubular neighbourhood
  of~\(\piinvdot{Q}\) in~\(\piinvdot{P}\cup\piinvdot{Q}\) with quotient~\(\bar V=V/T\).
  Then the restriction map 
  for this pair of faces factorizes as
  \begin{equation}
    H^{*}(P)
    \stackrel{\cong}\longrightarrow
    H^{*}(\dot P) \longrightarrow
    H^{*}(\bar V\setminus\dot Q)
    \stackrel{\cong}\longleftarrow
    H^{*}(\bar V)
    \stackrel{\cong}\longrightarrow
    H^{*}(\dot Q)
    \stackrel{\cong}\longleftarrow
    H^{*}(Q).
  \end{equation}
  Since \(V\) can be taken as a component of~\(U\) in the definition of
  equivariant injectivity, this shows that condition~\ref{cc4} is
  equivalent to~\ref{cc3} for regular actions.

  We claim that if \(T\) does not act regularly on~\(X\), then
  for any fixed~\(K\), conditions~\ref{cc2} and~\ref{cc4} are still equivalent.
  To see this,
  consider
  the blow-up \(\BlL Z\) of (all characteristic submanifolds of)
  a component~\(Z\) of~\(X^{K}\). We have
  \begin{equation}
    \depth_{R_{L}} H_{L}^{*}(\BlL Z) = \depth_{R_{L}} H_{L}^{*}(Z)
  \end{equation}
  by Proposition\nobreakspace \ref {thm:depth-syzygy-blowup}\,\ref{thm:depth-blowup}.
  By what we have just proven,
  this value is positive if and only if \(Y\) has no essential basic cohomology.
  Now look at a single characteristic submanifold~\(Y\) of~\(Z\) and
  its preimage~\(\tY\) in the blow-up~\(\tilde Z\) of~\(Z\) along~\(Y\). 
  The blow-up of an \(L\)-stable tubular neighbourhood of~\(Y\) is a tubular neighbourhood of~\(\tY\),
  and the free parts of them are equivariantly diffeomorphic.
  Hence
  \(Z\) has no essential basic cohomology if and only if this is true for~\(\tilde Z\)
  and by induction also for~\(\BlL Z\).
  This completes the proof.
\end{proof}

\section{Examples}
\label{sec:3:examples}

\subsection{Non-compact examples}
\label{sec:examples-non-compact}

Let \(Y\) be a locally standard torus manifold with face-acyclic orbit space,
for example a complete smooth toric variety.
Choose two fixed points~\(y_{1}\ne y_{2}\) of~\(Y\)
and consider the open torus manifold~\(X=Y\setminus\{y_{1},y_{2}\}\).

Let \(P=y_{1}\vee y_{2}\) be the minimal face of~\(Y/T\)
containing both \(y_{1}\) and \(y_{2}\).
Then
\begin{equation}
  H^{i}(\Bc^{*}(Q)) \cong
  \begin{cases}
    \kk & \text{if \(Q\ge P\) and \(i=1\),} \\
    0 & \text{otherwise if \(i>0\).}
  \end{cases}
\end{equation}
By Theorem\nobreakspace \ref {thm:condition-syzygy-HBc},
\(\HT^{*}(X)\) is a \((p-1)\)-st, but not a \(p\)-th syzygy, where \(p=\dim P\).

\goodbreak

The ``smallest'' example~\(X\) such that \(\HT^{*}(X)\) is an \((r-1)\)-st syzygy, but not free
is obtained by taking \(Y=(\CP^{1})^{r}=(S^{2})^{r}\), so that \(Y/T\) is an \(r\)-dimensional cube,
and removing two fixed points corresponding to opposite vertices.
For this space \(\HT^{*}(X)\) has been explicitly computed in~\citeorbitsone{Sec.~6.1}:
\begin{equation}
  \HT^{*}(X) \cong \bigoplus_{i=0}^{r-2} R[2i]^{\binom{r}{i}} \oplus K_{r-1}[2(r-1)],
\end{equation}
where \(K_{r-1}\) is the \((r-1)\)-st syzygy in the Koszul resolution of~\(R\),
\cf~\citeorbitsone{Sec.~2.4}.

\subsection{A big polygon space}
\label{sec:3:mutants}

Let \(X\) be the big polygon space~\(X_{1,1}(1,1,1)\) introduced in~\cite{Franz:maximal}.
This is the compact orientable manifold defined by the equations
\begin{alignat}{2}
  |u_{k}|^{2} + |z_{k}|^{2} &= 1 &\qquad & (k=1,2,3), \\
  u_{1} + u_{2} + u_{3} &= 0,
\end{alignat}
with~\(u\),~\(z\in\C^{3}\). The torus~\(T=(S^{1})^{3}\) acts on~\(X\)
by rotating the variables~\(z_{k}\) in the usual way.
We apply Theorem\nobreakspace \ref {thm:condition-syzygy-HBc} to verify that \(\HT^{*}(X)\)
is torsion-free, but not free.

As shown in~\cite[Sec.~7.1]{Franz:maximal}, \(X\) is equivariantly homeomorphic
to the \(7\)-dimensional mutant constructed in~\cite{FranzPuppe:2008}.
In particular, \(X/T\) is topologically a \(4\)-ball.
Now consider the standard \(T\)-action
on~\(\C^{3}\) and on its compactification~\(Y\cong S^{6}\). Then \(Y/T\)
is topologically a \(3\)-ball. Removing the centres \(p\in X/T\) and \(q\in Y/T\),
we obtain the Hopf fibration
\begin{equation}
  (X/T)\setminus\{p\} \to (Y/T)\setminus\{q\}.
\end{equation}
(This approach was used in~\cite{FranzPuppe:2008} to construct \(X/T\) and \(X\).)

The isotropy groups occurring in~\(X\) and~\(Y\) are the coordinate subtori of~\(T\).
For such a subtorus~\(K\) of rank~\(0<k\le r\), one has \(Y^{K}=S^{2(r-k)}\),
the one-point compactification of~\(\C^{2(r-k)}\),
and \(X^{K}\cong S^{2(r-k)}\times S^{1}\), where \(T/K\) acts trivially on~\(S^{1}\).
Hence \(P=X^{K}/T\) is a face of~\(X/T\) for~\(k<r\), and
\begin{equation}
  \label{eq:BcP-mutant}
  H^{i}(\B^{*}(P)) \cong
  \begin{cases}
    H^{*}(S^{1}) & \text{if~\(i=0\),} \\
    0 & \text{otherwise.}
  \end{cases}
\end{equation}
For \(k=r\) we obtain the fixed point set.
It consists of two circles, and \eqref{eq:BcP-mutant} holds for each of them.
For~\(Q=X/T\) the complex~\(\B^{*}(Q)\) 
has the following form:
\begin{equation}
  \B^{*}(Q)\colon
  \begin{array}{r|llll}
    1 & \kk^{2} & \kk^{3} & \kk^{3} & \kk \\
    0 & \kk^{2} & \kk^{3} & \kk^{3} & 0 \\
    \hline
    & 0 & 1 & 2 & 3
  \end{array}
  \qquad
  H^{*}(\B^{*}(Q))\colon
  \begin{array}{r|llll}
    1 & \kk & 0 & 0 & 0 \\
    0 & \kk & 0 & \kk & 0 \\
    \hline
    & 0 & 1 & 2 & 3
  \end{array}
\end{equation}
(The zeroth row computes \(H^{*}(S^{2})\) because \(\Hc^{*}(\dot Q)\) is concentrated in top degree;
the first row then follows from the fact that the spectral sequence
converges to \(H^{*}(Q)=\kk\).)
By Theorem\nobreakspace \ref {thm:condition-syzygy-HBc}, this shows that \(\HT^{*}(X)\) is a first, but not a second syzygy.

We can also compute \(\HT^{*}(X)\) independently of~\cite[Sec.~5]{FranzPuppe:2008}
and~\cite[Sec.~5]{Franz:maximal}.
Combining our calculation of the~\(H^{0}(\B^{*}(P))\) above
with Theorem\nobreakspace \ref {thm:iso-HABc-HBc}  and keeping in mind the identity~\eqref{eq:iso-RP-sum-RQ},
we get an isomorphism of graded vector spaces
\begin{equation}
  H^{i}(\AB^{*}(X)) \cong
  \begin{cases}
    R\oplus R[1]\oplus R[6]\oplus R[7] & \text{if \(i = 0\),} \\
    \kk  & \text{if \(i = 2\),} \\
    0 & \text{otherwise.}
  \end{cases}
\end{equation}
The summands~\(R[6]\) and~\(R[7]\) arise because we have
two terms~\(H^{*}(S^{1})\otimes R[6]\) corresponding
to the two fixed circles, but only one linear combination of them combines
with the rest to give \(R\oplus R[1]\) via~\eqref{eq:iso-RP-sum-RQ}.
If we choose the linear combination which restricts to the same value
on both fixed point components, then according to~\eqref{eq:iso-RP-RtF}
the above isomorphism is even one of \(R\)-modules.
(For~\(i=2\) there is no choice for how \(R\) acts.)

The augmented Atiyah--Bredon complex~\eqref{eq:3:atiyah-bredon} 
is the \(E_{1}\)~page
of a spectral sequence converging to~\(0\) \citeorbitsone{Lemma~5.6},
and it is exact at position~\(-1\) since \(\HT^{*}(X)\) is torsion-free.
The next differential~\(d_{2}\colon H^{0}(\AB^{*}(X))\to H^{2}(\AB^{*}(X))\) must therefore be non-zero,
whence
\begin{equation}
  \HT^{*}(X) \cong R\oplus\m[1]\oplus R[6]\oplus R[7].
\end{equation}

We can even describe the multiplicative structure of~\(\HT^{*}(X)\),
which was not determined in~\cite{FranzPuppe:2008}.

As a warm-up, we look at the face-acyclic torus manifold~\(Y\cong S^{6}\) instead of~\(X\).
Since \(\HT^{*}(Y)\) is concentrated in even degrees,
we have an isomorphism of~\(R\)-modules
\begin{equation}
  \label{eq:iso-HT-Y}
  \HT^{*}(Y) \cong H^{*}(S^{6})\otimes R \cong R \oplus R[6].
\end{equation}
To describe the cup product in this representation,
we use Proposition\nobreakspace \ref {thm:H0AB-face-ring}\,\ref{pr1}. 
We write the generators corresponding to the facets of~\(Y/T\) as~\(t_{1}\),~\(t_{2}\),~\(t_{3}\),
and \(u\),~\(v\) for the ones corresponding to the two fixed points.
Note that \(u v=0\).
Then \(t_{1} t_{2} t_{3}=u+v\) by the Masuda--Panov formula.
Choosing \(a=u-v\) as a generator of~\(R[6]\) in~\eqref{eq:iso-HT-Y},
we get
\begin{equation}
  \label{eq:relation-a2}
  a^{2}=u^{2}+v^{2}=(u+v)^{2}=(t_{1}t_{2}t_{3})^{2}.
\end{equation}
This determines the product completely since it is \(R\)-bilinear.

Inspection of the formula given in Proposition\nobreakspace \ref {thm:product-AB0}
tells us that the \(R\)-algebra \(H^{0}(\AB^{*}(X))\)
is the tensor product of~\(\HT^{*}(Y)\) and~\(H^{*}(S^{1})\). In other words,
it is the graded commutative \(R\)-algebra
on two generators~\(a\) and~\(b\) of degrees~\(6\) and~\(1\), respectively,
with the single relation~\eqref{eq:relation-a2}.
(The relation~\(b^{2}=0\) follows already from graded commutativity.)

The cohomology of the non-augmented Atiyah--Bredon sequence is the \(E_{2}\)~page of a spectral sequence
converging to~\(\HT^{*}(X)\). Hence \(\HT^{*}(X)\) is the kernel of the differential~\(d_{2}\) discussed above,
which is the subalgebra of~\(H^{0}(\AB^{*}(X))\) obtained by removing the degree~\(1\) component.
This means that the generator~\(b\)
cannot appear on its own, but must be accompanied by~\(a\) or one of the~\(t_{i}\).

\end{document}